\documentclass{article}
\usepackage{amsfonts, amsmath,amscd, amssymb, latexsym, mathrsfs, verbatim, wasysym } %stmaryrd
\usepackage{amssymb}
\usepackage{hyperref}
\usepackage[dvips]{graphicx}
\usepackage{epsfig}
\usepackage{psfrag}
\usepackage[colorinlistoftodos]{todonotes}
\usepackage[english]{babel}
\usepackage[utf8x]{inputenc}
\usepackage[all]{xy}
\usepackage{amscd}
\usepackage{tikz-cd}

\usepackage{appendix}
\usepackage{amsthm}
\usepackage{appendix}
\usepackage{mathtools}
\usepackage[dvips]{graphicx}
\usepackage{amsfonts}
\usepackage{amssymb}

\usepackage{epsfig}
\usepackage{psfrag}
\usepackage{chngcntr}
\usepackage{subfigure}
\usepackage{captcont}

\usepackage{color}
\definecolor{darkgreen}{cmyk}{1,0,1,.2}
\definecolor{m}{rgb}{1,0.1,1}
\definecolor{green}{cmyk}{1,0,1,0}
\definecolor{darkred}{rgb}{0.55, 0.0, 0.0}
\definecolor{test}{rgb}{1,0,0}
\definecolor{cmyk}{cmyk}{0,1,1,0}

\newcounter{diagram}
\numberwithin{diagram}{section}
\numberwithin{equation}{section}

\newtheorem{Equation}{}[section]

\newtheorem{example}[Equation]{Example}
\newtheorem{theorem}[Equation]{Theorem}
\newtheorem{proposition}[Equation]{Proposition}
\newtheorem{lemma}[Equation]{Lemma}
\newtheorem{corollary}[Equation]{Corollary}

\newtheorem{definition}[Equation]{Definition}

\newtheorem{remark}[Equation]{Remark}

\def\Av{\operatorname{Av}}

\def\Dom{\operatorname{Dom}}

\def\ind{\operatorname{ind}}
\def\End{\operatorname{End}}

\def\reg{\operatorname{reg}}

\def\Ext{\operatorname{Ext}}
\def\Kas{\operatorname{Kas}}
\def\Prop{\operatorname{Prop}}

\def\Im{\operatorname{Im}}
\def\Ind{\operatorname{Ind}}

\def\Supp{\operatorname{Supp}}

\def\pr{\operatorname{pr}}
\def\id{\operatorname{id}}

\def\Range{\operatorname{Range}}
\def\supp{\operatorname{supp}}
\def\Ind{\operatorname{Ind}}
\def\ind{\operatorname{ind}}
\def\End{\operatorname{End}}

\def\maB{\mathcal{B}}

\def\maD{\mathcal{D}}
\def\maE{\mathcal{E}}
\def\maH{\mathcal{H}}
\def\maK{\mathcal{K}}

\def\maL{\mathcal{L}}
\def\maM{\mathcal{M}}

\def\maP{\mathcal{P}}
\def\maR{\mathcal{R}}

\def\maF{\mathcal{F}}
\def\maU{\mathcal{U}}
\def\maV{\mathcal{V}}

\def\del{\partial}

\def\C{\mathbb C}
\def\D{\mathbb D}
\def\F{\mathbb F}
\def\R{\mathbb R}
\def\S{\mathbb S}
\def\Z{\mathbb Z}
\def\N{\mathbb N}

\def\maB{{\mathcal B}}

\def\maE{{\mathcal E}}
\def\maF{{\mathcal F}}
\def\maM{{\mathcal M}}

\def\maH{{\mathcal H}}

\def\maU{{\mathcal U}}

\def\maS{{\mathcal S}}

\def\what{\widehat}

\def\pa{\partial}

\def\ep{\epsilon}

\usepackage{color}
\definecolor{darkgreen}{cmyk}{1,0,1,.2}
\definecolor{m}{rgb}{1,0.1,1}
\definecolor{green}{cmyk}{1,0,1,0}

\definecolor{test}{rgb}{1,0,0}
\definecolor{cmyk}{cmyk}{0,1,1,0}

\def\diam{\operatorname{diam}}

\def\Dom{\operatorname{Dom}}

\def\ind{\operatorname{ind}}
\def\End{\operatorname{End}}

\def\redg{\operatorname{red}}
\def\reg{\operatorname{reg}}

\usepackage[top=1in, bottom=1.25in, left=1.25in, right=1.25in]{geometry}% removing this results in a non-centered text

\begin{document}

\title{The Higson-Roe sequence for \'etale groupoids \\ I. Dual algebras and compatibility with the BC map}

 %\author{Moulay-Tahar Benameur and Indrava Roy}
\author{Moulay-Tahar Benameur\\{IMAG, Univ Montpellier, CNRS, Montpellier, France}\\\textit{Email address}: \texttt{moulay.benameur@umontpellier.fr} \and
Indrava Roy\\Institute of Mathematical Sciences, {HBNI}, Chennai, India\\\textit{Email address}:\texttt{indrava@imsc.res.in}}
%\address{IMAG, Univ Montpellier, CNRS, Montpellier, France}
%\email{moulay.benameur@umontpellier.fr}

%\author[I. Roy]{Indrava Roy\\Institute of Mathematical Sciences, Chennai, India\\Email:indrava@imsc.res.in}
%\address{Institute of Mathematical Sciences, Chennai, India}
%\email{indrava@imsc.res.in}

\maketitle
\begin{abstract}
We introduce the dual Roe algebras for proper \'etale groupoid actions and deduce the expected Higson-Roe short exact sequence.
When the action is co{-}compact, we show that the Roe $C^*$-ideal of locally compact operators is Morita equivalent to the reduced $C^*$-algebra of our groupoid,
and we further identify the boundary map of the associated periodic six-term  exact sequence with the Baum-Connes map, via a  Paschke-Higson map for groupoids.
For proper actions on continuous families of manifolds of bounded geometry, we associate with any $G$-equivariant Dirac-type family,
a coarse index class which generalizes the Paterson index class and also the Moore-Schochet Connes' index class for laminations.\\
\end{abstract}


\section{Introduction}\label{Preliminaries}

This paper is a first of a series of articles where we systematically investigate the expected universal  Higson-Roe analytic surgery exact sequence  for Hausdorff \'etale groupoids. This first paper is dedicated to the introduction of the dual Roe algebras for proper groupoid actions, and to the identification of the boundary maps appearing in the associated periodic $K$-theory exact sequence, yielding to the notion of coarse $G$-index for Paterson's continuous $G$-families of bounded geometry manifolds.
%Functoriality properties as well as the proof of the universal Higson-Roe sequence for \'etale groupoids  are postponed to the subsequent parts of this series where some applications will also be given.
% More precisely, the Paschke duality for countable discrete group actions is proved in \cite{BenameurRoy2} and we then deduce functorilaity as well as the allowed universal exact sequence. The Paschke-Higson duality theorem in the general case turns out to deserve more care and is  only treated  in the third part of this series with some applications to topology and geometry.

The $K$-theory index map can nowadays be more efficiently defined using the language of \'etale groupoids with their actions on spaces, mostly manifolds \cite{ConnesBook}.
For a countable discrete group $\Gamma$ for instance, this correspondence goes back to the work of Mischenko and Kasparov and yields for any  $\Gamma$-cover $\Gamma - {\what M}\to M$ over a
closed manifold $M$, to a map from the $K$-homology group of $M$ to the $K$-theory of the $C^*$-algebra of the group $\Gamma$:
$$
\Ind_\Gamma \; : \; K_* (M) \longrightarrow K_*(C^*_r\Gamma).
$$
Any such $\Gamma$-cover is uniquely determined (up to isomorphism)  by a (homotopy class of a) continuous map from $M$ to the classifying space $B\Gamma$ with its universal $\Gamma$-cover $E\Gamma\to B\Gamma$, and one can assemble all these maps into a universal map \cite{BaumConnes}:
$$
\mu_\Gamma : RK_*^\Gamma (E\Gamma) \longrightarrow K_*(C_r^*\Gamma),
$$
where $RK_*^\Gamma (E\Gamma)\simeq RK_* (B\Gamma)$ stands for the topological $K$-homology group of $B\Gamma$, defined using the inductive system of cocompacts closed subspaces,
while $K_*(C_r^*\Gamma)$ is the usual topological $K$-theory of the reduced $C^*$-algebra of $\Gamma$. This assembly map is known, when $\Gamma$ is torsion free, to be an isomorphism
for a large class of groups, and the so-called Baum-Connes conjecture states that this assembly map
should be an isomorphism for all torsion-free countable discrete groups. When the group has torsion, then  one has to replace the universal space $E\Gamma$  by
the universal space for proper $\Gamma$-actions, usually denoted ${\underline E}\Gamma${{, and the Baum-Connes conjecture again states that the similar assembly map
$\mu_\Gamma : RK_*^\Gamma ({\underline{E}}\Gamma) \to K_*(C_r^*\Gamma)$ is an isomorphism}}. For more details on this conjecture, see for instance  \cite{BaumConnesHigson,BaumConnes}.
\\

In their seminal work on {\em{``Mapping Surgery to Analysis''}} \cite{HRI05, HRII05, HRIII05}, N. Higson and J. Roe proved that the assembly map  fits in a six-term periodic exact sequence as a boundary map,
this allowed them to obtain their universal periodic exact sequence for any such $\Gamma$. In analogy with the deep fundamental exact sequence of surgery \cite{Wall},
N. Higson and J. Roe called their sequence the analytic surgery exact sequence and they in particular introduced an analytic structure
group which plays the role of the structure group and is an obstruction group which  vanishes whenever the group $\Gamma$ satisfies the Baum-Connes conjecture.
These results rely in particular on the Voiculescu theorem for ample representations \cite{Voiculescu} and also on the Paschke-Higson duality theorem which extends
Poincar\'e duality to the non-smooth setting \cite{Paschke, HigsonPaschke}.
By using signature operators,  N. Higson and J. Roe went further and constructed an explicit ``commutative diagram'' relating the fundamental
sequence of surgery with their analytic exact sequence, so {\em{bringing surgery to analysis}}. In the same lines, P. Piazza and T. Schick proved later on a
similar commutative diagram relating now the Stolz exact sequence with the Higson-Roe sequence, and deduced  some important results on the moduli
space of metrics of positive scalar curvature \cite{PiazzaSchick}. Important applications to
rigidity conjectures of reduced eta invariants  have then been efficiently explored with the help of these new tools, see for instance \cite{HigsonRoe2010, PiazzaSchick, BenameurRoyJFA}. \\

On the other hand,
%the slightly more general  Baum-Connes assembly map with coefficients in an extra $\Gamma$ $C^*$-algebra $A$, can be defined and is indeed needed in many constructions and proofs. It can be written as
%$$
%\mu_\Gamma^A : RK_{*}^\Gamma ({\underline E}\Gamma, A) \longrightarrow K_* (A\rtimes_r \Gamma).
%$$
%Again it is known to always be an isomorphism for a large class of countable discrete groups, see for instance \cite{BaumConnesHigson}.
when $\Gamma$ acts on a locally compact Hausdorff space $Z$, the assembly map corresponding to this topological dynamical system can be defined using a natural extension of
$\mu_\Gamma$ where one adds the $C^*$-algebra $C_0(Z)$ as {\em{coefficients}}. This natural generalization of the Baum-Connes map can actually
be interpreted as the universal index map  for an \'etale groupoid, namely the transformation groupoid $Z\rtimes \Gamma$.
More generally, there is a well defined assembly map for any locally compact (\'etale) groupoid $G$ which uses a locally compact model for
the classifying space ${\underline E}G$ of proper groupoid actions \cite{Tu, LeGall}. If the unit space is denoted $X=G^{(0)}$, then the Baum-Connes assembly map for $G$ is a map:
$$
\mu_G: RK_*^G ({\underline E}G , X) \longrightarrow K_*(C^*_rG),
$$
where the LHS, {{say the group $RK_*^G ({\underline E}G , X)$,}} is again defined using an inductive system for the classifying space ${\underline E}G$ of proper $G$-actions,
now a limit of {{bivariant $KK$-groups $KK_G(Y, X)$ over cocompact $G$-subspaces $Y$ of ${\underline E}G$}}. The RHS  is again the $K$-theory of the reduced $C^*$-algebra of the groupoid $G$. With no surprise, the construction of the assembly map relies again
on the index theory for groupoid actions, see \cite{ConnesBook}. Notice that the language of groupoid actions allows to encompass the case of discrete countable groups as well as that of foliations and even laminations. \\

{{Following the Higson-Roe program, the}} next step is to introduce the six-term exact sequence for proper actions of \'etale groupoids {{which would incorporate  the previous Baum-Connes assembly map for $G$ as a boundary map}}. It is our goal here to  extend the Higson-Roe constructions and introduce  the dual Roe algebras for proper groupoid actions on families of metric
spaces $\rho: Y\to G^{(0)}$, so as to encompass new geometric situations. Given such  a $G$-proper space $(Y, \rho)$ such that the anchor map $\rho$ is open, and given a Hilbert $G$-module  $\maE$ which is endowed with a non-degenerate $G$-equivariant representation of the $G$-algebra $C_0(Y)$, the notion of (uniform) propagation of operators on $\maE$, with respect to this representation and to the proper family metric on $Y$, is introduced and
we hence define the dual Roe algebras in the same lines as for groups, and easily obtain the corresponding short exact sequence of dual $C^*$-algebras
$$
0 \to C^*_G (Y, \maE) \hookrightarrow D^*_G (Y, \maE) \rightarrow Q^*_G (Y, \maE) \to 0.
$$
Our candidate six-term exact sequence {{is then deduced by applying the topological $K$-functor and by using Bott periodicity:}}
\begin{displaymath}\label{Figure1}
\xymatrixcolsep{1pc}\xymatrix{
K_{*} (C^*_G (Y, \maE)) \ar[rr]   &  & K_{*} (D^*_G (Y, \maE))   \ar[dl]^{}\\ & K_*(Q^*_G (Y, \maE))  \ar[ul]_{\partial_*}  }
\end{displaymath}
 As an important application,
we consider the Paterson category of families of smooth bounded geometry manifolds with $G$-actions and
we deduce that the coarse $G$-index of $G$-invariant families $\maD$ of  fully  elliptic operators acting on the sections of a $C^{\infty, 0}$-bundle $E$ is well defined:
$$
\Ind_G (\maD) \; \in \; K_{*} (C^*_G (Y, \maE_{Y, E})).
$$

{{The next important result proved in the present paper is the compatibility of the boundary map $\partial_*$ with the Baum-Connes map for the groupoid $G$. To this end, one has to choose specific Hilbert modules which are naturally associated with $Y$.
More precisely, given a full $G$-equivariant $\rho$-system $(\mu_x)_{x\in X}$ on $Y$, and denoting by  $\maE_{Y'}$ the Hilbert $G$-module corresponding to the continuous field of Hilbert spaces over $X$ given by $\left(L^2(Y_x, \mu_x)\otimes \ell^2(G^x)\right)_{x\in X}$,  we obtain the following theorem.}}

\begin{theorem}\ Assume that $Y$ is $G$-compact, then the dual Roe algebra $C^*_G (Y, \maE_{{{Y'}}})$ is Morita equivalent to the reduced $C^*$-algebra of the groupoid $G$. In particular, we have an isomorphism
$$
\maM_* : K_{*} (C^*_G (Y, \maE_{{Y'}})) \longrightarrow K_* (C^*_r G).
$$
\end{theorem}
When the action of $G$ on $Y$ is for instance free, then this theorem holds already with the Hilbert $G$-module $\maE_Y$ associated with the continuous field of Hilbert spaces $\left(L^2(Y_x, \mu_x)\right)_{x\in X}$. Using the isomorphism, we see that our class $\Ind_G (\maD)$  extends to the non-cocompact bounded geometry case,
the  Paterson $G$-index and hence in the case of bounded geometry laminations the  Moore-Schochet Connes' index class, {{which was  previously only defined when the ambient space is compact.}}
In the presence of a fiberwise $G$-invariant metric of uniformly positive scalar curvature and under the usual spin assumption, we also define a secondary class living in
the structure group $K_{*} (D^*_G (Y, \maE_{Y, E}))$, {{extending results of Roe-Higson and Piazza-Schick \cite{HRIII05, HigsonRoe2010, PiazzaSchick}.}} \\

On the other hand, the Paschke-Higson morphism can be defined for any Hilbert $G$-module $\maE$ as above with the non-degenerate $G$-equivariant representation of $C_0(Y)$, and is  now valued in  a $G$-equivariant $KK$-group \cite{LeGall}:
$$
\maP_* : K_*(Q^*_G (Y, \maE)) \longrightarrow KK^{*+1}_G (Y, X),
$$
The  compatibility theorem  can then be stated as follows:%\footnote{The case  when $G$ is a discrete countable group  has also been studied  in the recent article \cite{GWY16}}

\begin{theorem}\
For $*=0$ and $*=1$, the following diagram commutes:
\[
\begin{CD}\label{BaumConnesassembly}
K_*(Q^*_G(Y,\maE_{{Y'}})) @> \del_*    >> K_{*+1}(C^*_G(Y,\maE_{{Y'}})) \\
@V\maP_* VV    @V\maM_{*+1}VV \\
KK^{*+1}_G(Y,X) @> \mu^{*+1}_{Y} >> K_{*+1}(C^*_{\redg}G)\\
\end{CD}
\]
\end{theorem}
When $G=X$ is a space groupoid such that $X$ is compact and metrizable, our results  are closely related with the classical results of Pimsner-Popa-Voiculescu \cite{PPV}. The case when $G$ is a discrete countable group has also been studied in the recent article \cite{GWY16} as well as  in \cite{Zenobi}, we thank the referee for pointing out  Zenobi's paper to us. {{For {\underline{free}} and proper discrete group actions, Zenobi showed, using the Paschke isomorphism, that his structure group,  roughly speaking the homotopy fiber of the assembly map, is isomorphic to the Higson-Roe structure group.  In fact, Zenobi went further and introduced in his thesis a structure group even for  Lie groupoids while our definition in the present paper is the exact extension of the Higson-Roe approach to the category of locally compact groupoids. It is therefore an interesting task to compare the Zenobi structure  group with ours in the case of Lie groupoids. There  always exists an obvious group morphism from our structure group to the Zenobi one, and if the groupoid is torsion-free and satisfies Paschke-Higson duality, then one easily shows that this morphism is an isomorphism, but the general case remains to be investigated.}} Finally, in order to prove the  universal Higson-Roe exact sequence and the corresponding universal commutative diagram, further properties of the Paschke-Higson map are needed. {In order} to keep this paper in a {reasonable} size, these properties together with some geometric corollaries, will be  investigated in the next articles of this series.

\tableofcontents

\medskip

{\bf{Preliminaries and notations.}} The groupoid $G$ will be a locally compact Hausdorff \'{e}tale groupoid. We shall denote by  $X:=G^{(0)}$ the space of units of the groupoid $G$ and by $G^{(1)}$ the space of arrows of $G$. Then $X$ is identified with a closed (and open) subspace of $G^{(1)}$ and we shall  sometimes also denote by $G$ the space $G^{(1)}$ of arrows of the groupoid $G$.
The source and range maps are denoted  $s$ and $r$ respectively and are then \'etale maps from $G^{(1)}$ to $G^{(0)}$. Given subsets $A$ and $B$ of $X$, we denote by $G_A$ and $G^B$ the subspaces of $G$ defined as $s^{-1} A$ and $r^{-1} B$ respectively. The intersection $G_A\cap G^B$ is denoted $G_A^B$ and if $A=\{x\}$ then we denote $G_A$ as simply $G_x$, and similarly for the obvious notations $G^x$ and $G_x^{x'}$.
%We refer for instance to \cite{MoerdijkMrcun} (M2I: HERE THERE MUST BE AN OLDER REFERENCE) for the basic theory of \'etale groupoids and to \cite{LeGall} for most of the properties of $G$-algebras and Hilbert $G$-modules reviewed in this section.

Given a locally compact Hausdorff space $Z$, we denote by $C_0(Z)$ the $C^*$-algebra of continuous complex valued functions vanishing at infinity. As usual,  $C_b(Z)$ denotes the $C^*$-algebra of bounded continuous functions on $Z$. For simplicity, and {in order} to avoid some annoying technicalities, we shall assume that  the {Tietze} theorem applies for all our spaces and all commutative $C^*$-algebras will have countable approximate units.
If $A$ is a given $C^*$-algebra and $\maE$ is a Hilbert $A$-module, then we denote by $\maL_A (\maE)$ the $C^*$-algebra of adjointable operators on $\maE$,
while $\maK_A (\maE)$ is the ideal of $A$-compact operators, see \cite{Kasparov, Lance} for more details on the properties of adjointable operators. \\

{\em{ Acknowledgements.}}
The authors  wish to thank P.S. Chakraborty, T. Fack, N. Higson, V. Mathai, P.-E. Paradan, P. Piazza, B. Saurabh, G. Skandalis, R. Willett and V. Zenobi
for many helpful discussions.
MB thanks the French National Research Agency for support via the ANR-14-CE25-0012-01 (SINGSTAR).
IR thanks the Homi Bhabha National Institute, the Indian Statistical Institute Delhi, and the Indian Science and Engineering Research Board via MATRICS project MTR/2017/000835 for support.

%
%\begin{lemma}
%Let $\pi: C_0(Y)\rightarrow \maL_{C_0(X)}(E)$ be a $G$-equivariant representation. For $f\in C_0(Y)$, we have $\pi(f)\in \maL_{C_0(X)}(E)^G$.
%\end{lemma}
%
%\begin{proof}
%One checks that $\pi(f)\otimes_s id= (s^*\pi)(s^*f)\in \maL_{C_0(G)}(s^*E)$, and similarly $\pi(f)\otimes_r id= (r^*\pi)(r^*f)\in \maL_{C_0(G)}(r^*E)$. Also, $\alpha(s^*f)= r^*f$ (as multipliers in $M(r^*C_0(Y))$, extending $\alpha$ to an isomorphism $ M(s^*C_0(Y))\xrightarrow{\cong} M(r^*C_0(Y))$ still denoted by $\alpha$). Using the definitions of a $G$-invariant element and a $G$-equivariant representation, it is easy to prove the assertion.
%\end{proof}

\section{Dual algebras for \'{e}tale groupoids}\label{Roealgebra}

We shall freely use the results given in Appendix \ref{AppendixA}, this is an overview of some needed constructions on $G$-algebras, Hilbert $G$-modules and $G$-representations.

\subsection{$G$-spaces}

The first class of  $G$-algebras that we shall use in this paper is given by the commutative ones.

\begin{definition}[$G$-space]
\label{Gspace}
A (Hausdorff) topological space $Y$ is a (right) $G$-space if we are given:
\begin{enumerate}
\item a continuous  map $\rho: Y\rightarrow X$, called the anchor map;
\item a continuous map $\lambda:  Y\rtimes_r G\rightarrow Y$ such that:
\begin{itemize}
\item $\rho(\lambda(y,\gamma))=s(\gamma)$ and $\lambda(y,x)=y$, $x{{\in X}}$ is identified with its image in $G^{(1)}$;
\item $\lambda(\lambda(y,\gamma),\gamma')=\lambda(y,\gamma\gamma'), \text{ if } (\gamma,\gamma')\in G^{(2)} \text{ with }\rho(y)=r(\gamma)$,
\end{itemize}
where $ Y\rtimes_r G:= \{(y,\gamma)\in Y\times G /  \rho(y)= r(\gamma) \}$.
\end{enumerate}
\end{definition}
We shall write $\lambda(y,\gamma)$ as simply $y\gamma$ and refer to the $G$-space $(Y, \rho, \lambda)$ simply as $(Y, \rho)$ or sometimes as just $Y$.  Associated with any such $G$-space, there is an equivalence relation $\sim$ defined by
$$
y\sim y' \Longleftrightarrow \left[ \exists \gamma\in G^{\rho(y)}, y\gamma = y' \right].
$$
The equivalence classes are also called $G$-orbits and the quotient space is then endowed with its quotient topology.

%{\{SAY SOME HERE WORDS ABOUT THE CONDITION OPEN SURJECTIVE ON $\rho$? SHOULD PROBABLY BE: IF OPEN AND SURJECTIVE THEN CONTINUOUS FIELD IN THE SENSE OF DIXMIER? REALLY?}}{I2M: Not sure...}

\begin{remark}
We may assume in the previous definition that $\rho$ is surjective since only the $G$-saturated subspace $\rho(Y)$ of $X$ and the subgroupoid $G_{\rho (Y)}^{\rho (Y)}$ would be involved. Moreover, in most of the interesting examples for us, the anchor map is open and surjective. We shall therefore restrict ourselves to the case of an open surjective anchor map so as to avoid  u.s.c. fields which are not continuous in the sense of Dixmier.
\end{remark}

The following proposition is then standard and the proof is a straightforward verification which is omitted.
\begin{proposition}
\label{isopullback}
Let $Y$ be a $G$-space as above. Then $C_0(Y)$ is a $G$-algebra, more precisely:
\begin{enumerate}
\item We have $C_0(G)$-algebra identifications
$$
s^*C_0(Y)\cong C_0(Y\rtimes_s G)\text{ and } r^*C_0(Y)\cong C_0(Y\rtimes_r G).
$$
\item The $G$-algebra structure $\alpha_Y: s^*C_0(Y)  \longrightarrow r^*C_0(Y)$ is defined by
$$
\alpha_Y(\varphi) (y, g) = \varphi (yg, g), \text{ for }\varphi\in C_c(Y\rtimes_s G).
$$
\end{enumerate}
\end{proposition}

\begin{remark}
 The proof   of Proposition \ref{isopullback} uses the fact that the restriction maps $ C_0(Y\times G)\rightarrow C_0(Y\rtimes_s G)$ and $ C_0(Y\times G)\rightarrow C_0(Y\rtimes_r  G)$ induce the announced isomorphisms.
\end{remark}

%From the previous description of $s^*C_0(Y)$ and $r^*C_0(Y)$, we can easily write the $G$-action.

A locally compact  Hausdorff $G$-space $Y$ is called a  proper $G$-space (we also say that the $G$-action is proper)  when the map
$$
Y\rtimes_r G \longrightarrow Y \times Y \text{ given by } (y,\gamma)\longmapsto (y\gamma, y),
$$
is proper. The $G$-space $Y$ is a free $G$-space (we also say that the $G$-action is free) if the above map is injective, i.e. for any $y\in Y$, the isotropy group
$$
G(y) :=\{\gamma\in G^{\rho(y)}, y\gamma = y\},
$$
is reduced to $\{\rho (y)\}$.
The space of orbits for a proper $G$-action on $Y$ is then Hausdorff. A proper locally compact Hausdorff $G$-space $Y$ is $G$-compact if the quotient space of orbits is  compact.

\begin{definition}[Cutoff function]
\label{cutoff}
A cutoff function for a $G$-space $(Z,\rho)$ is a continuous  map $c: Z\rightarrow [0, 1]$ such that:
 $$
 \sum_{g \in G^{ \rho (z)}} c(zg)=1, \quad \forall z \in Z.
 $$
 and such that for any compact subspace $K$ of $Z$ the space $(K\cdot G) \cap \Supp (c)$ is compact.
%\item $s: supp(c\circ r)\rightarrow X$ is proper.
\end{definition}

We recall that proper $G$-spaces always  have cutoff functions \cite{Tu}, in the $G$-compact case, these cutoff functions are then compactly supported.

%\begin{remark}
%$Y\rtimes_r G$ is a closed subspace of $Y\times G$ and is endowed with the usual structure of a groupoid, see comment after Proposition \ref{EY}.
%\end{remark}
%

\subsection{Metric $G$-spaces and Roe algebras}\label{metricGspace}

All the $G$-spaces that will be considered are associated with $G$-equivariant continuous fields of commutative $C^*$-algebras over $X$ in the sense of Dixmier. This imposes for such $G$-space $(Y, \rho)$ that  the anchor map  $\rho$ {{is}} open. As explained above, we shall assume  that it is also surjective.

\begin{definition}
A locally compact (right) $G$-space $(Y, \rho)$ with the anchor map $\rho:Y\to X$ is a {\em{$G$-family of proper metric spaces}} if  we are given a continuous scalar valued function $d_Y$ on the closed subset $\{(y, y')\in Y^2, \rho (y)= \rho (y')\}$ of $Y^2$ such that
\begin{enumerate}
\item For any  fiber  $Y_x:=\rho^{-1} (x)$ of $\rho$, the restriction $d_x$ of $d_Y$ to $Y_x \times Y_x$  is a distance which defines the induced topology of $Y_x$.
\item (properness) {{Any closed bounded subset $Z$ of $Y$ such that $\rho (Z)$ is compact in $X$, is itself compact in $Z$. Boundedness means that  $\sup_{x\in X} \diam_{d_{\rho(y)}} (\rho^{-1} (x) \cap Z) < +\infty$. }}
%, the restriction  compact subset $X'$ of $X$, the res$Z$ of $Y$ and any $R>0$, the closed $R$-ball neighborhood of $Z$, defined by $\{y\in Y, d_{\rho(y)} (y, Z\cap \rho^{-1}(\rho(y)))\leq R\}$, is compact.
%restriction of $d_Y$ to $\rho^{-1}(X')\times_\rho \rho^{-1}(X')$ is proper.
%  closed subset $K$ of $Y$ such that the map $x\longmapsto  \diam (\rho^{-1} (x) \cap K) < +\infty$ is locally bounded, the restriction of $\rho$ to $K$ is proper.
\item (invariance)  For any $g\in G$ and any $(y,y')\in Y_{r(g)}^2$, we have $d_{s(g)} ( yg, y'g) = d_{r(g)} (y, y')$.
\end{enumerate}
\end{definition}
In his proof of the Novikov conjecture for (hyper)bolic groupoids \cite{Tu}, Jean-Louis Tu introduced in the general framework of topological spaces the notion of {\em a continuous family of metric spaces} together with an isometric action of a (\'etale) groupoid $G$ by imposing the first and third axioms above. In our case, we restrict ourselves to the locally compact case and we moreover impose the properness axiom which is the usual condition  needed to define the coarse Roe algebras.  So the terminology ``$G$-family'' includes already that $G$ acts isometrically. Since $G$ acts properly on $Y$, an easy argument shows indeed that any metric structure $d_Y$ satisfying the first and second axioms gives rise a metric structure $d'_Y$ which also satisfies the last axiom.

%In the properness axiom, $\diam (\rho^{-1} (x) \cap K)$ is the usual diameter given by
%$$
%\diam (\rho^{-1} (x) \cap K) := \sup \{d_x(y, y'), (y, y')\in [\rho^{-1} (x) \cap K]^2\}.
%$$
 If $Z_1, Z_2$ are two subspaces of $Y$, then the distance $d_Y(Z_1, Z_2)$ will be the fiberwise distance given when $\rho (Z_1)\cap \rho (Z_2)\neq \emptyset$ by
$$
d_Y(Z_1, Z_2) =\inf\{d_Y(z_1, z_2)\vert (z_1, z_2)\in Z_1\times Z_2, \rho(z_1)=\rho(z_2)\}.
$$
So, if $\rho (Z_1)\cap \rho (Z_2) = \emptyset$, then the convention is $d_Y(Z_1, Z_2) = +\infty$.

%
%\begin{lemma}
%A proper locally compact (Hausdorff and paracompact) $G$-space $(Y, \rho)$ always admits a structure of a continuous family of $G$-invariant metric spaces. ARE THEY ALL EQUIVALENT?
%\end{lemma}
%
%
%
%\begin{proof}
%
%M2I: MISSING PROOF HERE. SHOULD BE STRAIGHTFORWARD.
%
%\end{proof}
%

%\subsection{Roe algebras}

Recall that the continuous anchor map $\rho$ is open and surjective and that $(Y, d_Y)$ is a $G$-family of proper metric spaces.
Let $E$ be a Hilbert $G$-module which corresponds to the continuous field of Hilbert spaces $(H_x)_{x\in X}$ over $X$ in the sense of Dixmier. So,  any element $g \in G$ yields a unitary isomorphism $V_g: H_{s(g)}\xrightarrow{\cong} H_{r(g)}$ with the relation over $G^{(2)}$ recalled in Appendix \ref{AppendixA}.
Consider now a non-degenerate $G$-equivariant $C_0(X)$-representation $\pi: C_0(Y) \rightarrow \maL_{C_0(X)}(E)$.

\begin{definition}[Finite propagation]
An operator $T\in \maL_{C_0(X)}(E)$ has finite propagation with respect to $d_Y$ if there exists a constant $R >0$ such that
$$ \pi(f)T\pi(g)=0 \quad \text{whenever} \quad d_{Y}(\supp(f), \supp(g))>R.
$$
for all $f,g\in C_0(Y)$. The least such constant $R>0$ is the \emph{propagation} of $T$.
\end{definition}

\begin{remark}
Since $\pi$ is non-degenerate, it can be extended to a $*$-homomorphism $\bar\pi: C_b(Y) \to \maL_{C_0(X)}(E)$.
 %\cite{Lance}[Proposition 2.5].
 In the previous definition, we may then equivalently use functions $f, g$ from $C_b(Y)$ and the extended representation $\bar\pi$.
%Notice that we ask here for finite fiberwise propagation.
\end{remark}

%{I2M: Added Definition (\ref{supp}) and Remark (\ref{suppremark}).}

%\begin{remark}\label{supp}
%One can define as usual the support of any $T\in \maL_{C_0(X)}(E)$
%% to be:
%%\begin{eqnarray*}
%%\Supp(T)&:=& \{(y,y')\in Y\times_\rho Y| \text{ for every open neighbourhoods } U_y, U_{y'} \text{ of } y,y' respectively, \\
%% &&\text{ there exist functions } \phi\in C_0(U_y), \phi'\in C_0(U_{y'}) \text{ with } \pi(\phi)T\pi(\phi') \neq 0 \}
%%\end{eqnarray*}
%then it has finite propagation if and only if the diameter of its support is finite.
%\end{remark}

%{\{
%}}
%
%\begin{remark}\label{suppremark}
%If an operator $T \in \maL_{C_0(X)}(E)$ induces a family of operators given by fibrewise integral kernels $(K_{T_x})_{x\in X}$, then $\Supp(T)= \{(y,y') \in Y\times_\rho Y| (y,y')\in \Supp(K_{T_{\rho(y)}})\}$.
%\end{remark}

Recall the $C^*$-algebra $\left[\maL_{C_0(X)}(E)\right]^G$ of $G$-invariant adjointable operators on $E$ defined in the appendix. Recall also that $\maK_{C_0(X)}(E)$ is the ideal of $C_0(X)$-compact operators in the Hilbert module $E$.

\begin{definition}\label{Roe-Algebras}\
\begin{itemize}
\item The equivariant Roe-algebra $D^*_G(Y, E)$  is the norm closure (in $\maL_{C_0(X)}(E)$) of the space
$$
\{ T\in \left[\maL_{C_0(X)}(E)\right]^G,  T \text{ has finite propagation and } [T,\pi(f)]\in \maK_{C_0(X)}(E), \forall f\in C_0(Y)\}.
$$
\item The subspace $C^*_G(Y, E)$ is defined  as
$$
C^*_G(Y, E) := \{ T\in D^*_G(Y, E),  T\pi(f) \in \maK_{C_0(X)}(E), \forall f\in C_0(Y)\}.
$$
\end{itemize}
\end{definition}

\begin{remark}
Notice that by definition, finite propagation means  uniform finite propagation with respect to the $X$-variable.
\end{remark}

%For $T\in C^*_G(Y, E)$, one obviously also has  $\pi(f)T \in \maK_{C_0(X)}(E)$ for any $T\in C^*_G(Y, E)$ and any $f\in C_0(Y)$.
For any operator $T=(T_x)_{x\in X}$ of $D^*_G(Y, E)$, the operators $T_x$ all belong to the $C^*$-algebras $D^*(Y_x, E_x)$ and satisfy an obvious equivariance property.
If we consider  the $C^*$-algebras $D^*_X (Y, E)$  and $C^*_X (Y, E)$ defined using the pointwise groupoid structure of the space $X$, which are $G$-algebras,
then it is tempting to rather consider the $C^*$-algebras $D^*_X (Y, E)^G$  and $C^*_X (Y, E)^G$   composed of the $G$-invariant elements of  $D^*_X (Y, E)$ and $C^*_X (Y, E)$ respectively.
However,  and even for discrete countable groups, our definition \ref{Roe-Algebras}  better suits with the  Baum-Connes map as we shall see, see also \cite{HRII05}, \cite{PiazzaSchick}.

\begin{lemma}
For any Hilbert $G$-module $E$ with a non-degenerate $G$-equivariant $C_0(X)$-representation of the $G$-algebra $C_0(Y)$, $D^*_G(Y, E)$ is a (unital) $C^*$-algebra and $C^*_G(Y, E)$ is closed two-sided involutive ideal in $D^*_G(Y,E)$. Hence, denoting by $Q^*_G(Y, E)$ the quotient $C^*$-algebra,  we have the following short exact sequence of $C^*$-algebras:
$$
0\rightarrow C^*_G(Y, E)\longrightarrow D^*_G(Y, E)\longrightarrow Q^*_G(Y, E)\rightarrow 0.
$$
\end{lemma}

\begin{proof}\
If we admit that $D^*_G(Y, E)$ is a unital $C^*$-algebra, then it is  clear that $C^*_G(Y, E)$ is a closed two-sided involutive ideal in $D^*_G(Y, E)$, as a consequence of the fact that $\maK_{C_0(X)}(E)$ is a closed two-sided involutive ideal  in $\maL_{C_0(X)}(E)$. Hence, we only need to check that $D^*_G(Y, E)$ is a  (unital) $C^*$-algebra. Moreover, by standard arguments, only the proof that composition (and adjoint) of finite propagation operators are finite propagation operators. For elements $S,T\in D^*_G(Y,E)$ having finite propagations $R_S$ and $R_T$ respectively,  the propagation of the adjoint $T^*$ coincides with that of $R_T$, and the operator $ST$ also has finite propagation, in fact $\leq R_S+R_T$.
 The proof is standard.  Let for instance $f, g\in C_c(Y)$ be such that $d_Y(\supp (f), \supp (g)) > R_S+R_T$ and set $d_Y(\supp (f), \supp (g)) - (R_S+R_T) =2\ep >0$. Define the subsets $U_f=\{y\in Y, d_Y (y, \supp (f)) \leq R_T+\ep\}$ and $U_g=,\{y\in Y, d_Y (y, \supp (g)) \leq R_S+\ep \}$.  From the properness of  $d=d_Y$, we know that the subsets $U_f$ and $U_g$ are compact (disjoint) subspaces of $Y$. Let then $\varphi$ be a continuous compactly supported function on $Y$ such that
$$
\varphi\vert_{U_f} = 1 \text{ and } \varphi\vert_{U_g} = 0.
$$
So, $d_Y (\supp (f), \supp (1-\varphi)) \geq R_T+\ep > R_T$ and $d_Y (\supp (g), \supp (\varphi)) \geq R_S+\ep >R_S$. Hence, we may write
$$
\pi (f) TS\pi (g) = \pi(f) T \pi(\varphi) S \pi(g) + \pi(f) T \bar\pi(1-\varphi) S \pi(g) = 0.
$$
Therefore, $R_{TS}\leq R_T+R_S$ as announced.

%
%
%
%
%
%{Moreover, one has for any $f\in C_0(Y)$,
%$$
%[ST,\pi(f)]= S[T,\pi(f)]+[S,\pi(f)]T
%$$
%which shows that the commutator $[ST,\pi(f)]\in  \maK_{C_0(X)}(\maE)$. It is also evident that $D^*_G(Y,\maE)$ is closed under operator adjoints. Moreover, if $(T_n)_{n\geq 0}$ is a sequence of operators in $D^*_G(Y,\maE)$ such that $T_n\xrightarrow{n\rightarrow\infty} T$, then one can write
%$$
%[T_n,\pi_f]-[T,\pi(f)]= (T_n-T)\pi(f)-\pi(f)(T_n-T)
%$$
%thus $[T_n,\pi(f)]\xrightarrow{n\rightarrow\infty} [T,\pi(f)]$. As $[T_n,\pi(f)]$ belong to $\maK_{C_0(X)}(\maE)$ for every $n\geq0$, therefore so does $[T,\pi(f)]$. So $D^*_G(Y,\maE)$ is closed in the operator norm topology and thus a $C^*$-subalgebra of $\maL_{C_0(X)}(\maE)$. It is also straightforward to show that $C^*_G(Y,\maE)$ is a $*$-subalgebra of $D^*_G(Y,\maE)$.
%}
%{
%Moreover, given $S\in D^*_G(Y,\maE)$ and $T\in C^*_G(Y,\maE)$, one can show that $ST$ belongs to $C^*_G(Y,\maE)$. Indeed, one only needs to show that $TS\pi(f)\in C^*_G(Y,\maE)$ since all other properties follow from the definitions. This is seen by expressing $TS\pi(f)$ as follows:
%$$
%TS\pi(f)=T[S,\pi(f)]+(T\pi(f)).S
%$$
%}
%{Similar arguments show that $TS$ also belongs to $C^*_G(Y,\maE)$. This proves the second assertion in the lemma. }
\end{proof}

\begin{example}\label{Actions}
Take $G=X\rtimes \Gamma$ to be an action groupoid, where $\Gamma$ is a countable discrete group which acts by homeomorphisms on  a compact space $X$.
Consider a space $Y$   of the form $Y=X\times Z$, where $Z$ is a locally compact $\Gamma$-proper space, and take $E$ to be $C(X)\otimes H$ for a fixed unitary $\Gamma$-representation Hilbert space $H$ in which we also have a non-degenerate  representation of $C_0(Z)$. Then the $C^*$-algebra $\maL_{C(X)} (C(X)\otimes H)$ coincides with the  $C^*$-algebra $C(X, \maL(H)_{*-str})$ of  $*$-strongly continuous fields of operators. The Roe algebra
$D^*_G(Y, E)$ is then the closure of the space of (uniform) finite propagation $\Gamma$-equivariant  elements of $C(X, \maL(H)_{*-str})$.
\end{example}

\subsection{Relation with the groupoid  $C^*$-algebra}\label{MoritaIso}

Denote again by $\rho:Y\rightarrow X$ the open surjective continuous anchor map for the proper $G$-space $Y$.  Associated with the proper $G$-space $Y$, there is the  locally compact proper $G$-space $Y\times_\rho Y:=\{(y, y')\in Y^2, \rho(y)=\rho(y')\}$ which is also endowed with the compatible extra-structure of a groupoid for the rules
$$
s (y, y') =y', r(y, y')= y \text{ so that } (y, y') (y', y'') = (y, y'') \text{ and }(y, y') g = (yg, y'g)\text{ if } r(g)=\rho(y).
$$
There exists a {{full $\rho$-system} on   $Y$ \cite{Blanchard, Williams}, i.e. a  family  $(\mu_x)_{x\in X}$ where:
\begin{itemize}
\item  $\mu_x$ is a Radon measure on $\rho^{-1}(x)=Y_x$ whose support  is  $Y_x =\rho^{-1}(x)$.
\item (continuity)\; \; The map $x\mapsto \int_{Y_x} f(y) d\mu_x(y)$ is continuous for any $f\in C_c(Y)$.
\end{itemize}
{{A choice of such system allows to construct  a {\em{continous}} field of Hilbert spaces over $X$, or equivalently a $C_0(X)$-Hilbert module $\maE_Y$, as usual.}}
%We also fix a Borel measure $\Lambda$ on $X$ so that $Y$ carries a measure, denotes  $\mu_Y$, and defined by:
%$$
%\int_Y f(u) d\mu_Y(u) : = \int_X \int_{Y_x} f(y) d\mu_x(y) d\Lambda(x), \quad \text{ for }f\in C_c(Y).
%$$
The Hilbert $C_0(X)$-module  $\maE_Y$ corresponds to the {\em{continuous}} field of Hilbert spaces $(L^2(Y_x, \mu_x))_{x\in X}$. This is more precisely defined  by completing the pre-Hilbert $C_c(X)$-module
 $C_c(Y)$ of compactly supported continuous functions on $Y$ \cite{GengouxTuXu}. Recall that
$$
(\xi f)(y)=\xi(y)f(\rho(y)), \quad f\in C_c(X), \xi\in C_c(Y);
$$
and the $C_c(X)$-valued inner product is given by:
$$
<\eta,\xi>(x) = \int_{y\in Y_x} <\eta(y),\xi(y)> d\mu_{x}(y), \quad \eta,\xi\in C_c(Y).
$$
Notice that the map $y\mapsto <\eta (y), \xi (y)>$ then belongs to $C_c(Y)$ and hence by the continuity property of the Haar system, we deduce that $<\eta, \xi>$ belongs to $C_c(X)$. So we get in this way that the completion  $\maE_Y$  of $C_c(Y)$ with respect to the above pre-Hilbert $C_c(X)$-module structure, is a Hilbert $C_0(X)$-module. Notice that, under our assumptions, none of the Hilbert spaces $L^2(Y_x, \mu_x)$ is trivial and hence by a classical argument, the Hilbert module $\maE_Y$ is a full Hilbert module.

%Then any operator $T\in \maL_{C_0(X)} (\maE_Y)$ can be identified with a strong-$*$ continuous family of operators $(T_x)_{x\in X}$ on this field.

%\begin{lemma}
%\label{Hilbertfield}
%With the above notation we have an isomorphism of Hilbert spaces $H_x\cong L^2(Y_x,\mu_x)$, for any $x\in X$.
%\end{lemma}
%
%\begin{proof}
%Let $\xi\in C_c(Y)$. Consider the restriction $r_x(\xi)= \xi_{|_{Y_x}} \in C_c(Y_x)$. If $\xi\in I_x$, then we have:
%$$
%\int_{Y_x}<\xi_{|_{Y_x}}(y),\eta_{|_{Y_x}}(y)> d\mu_x(y)=0 \quad\text{ for any } \eta\in C_c(Y).
%$$
%
%In particular this means that $\xi_{|_{Y_x}}=0$. Thus there is a well-defined map $\psi:H_x\rightarrow L^2(Y_x,\mu_x)$ such that $\psi([\xi])=r_x(\xi)$. One can check that this is our desired isomorphism (in fact an isometric isomorphism of Hilbert spaces).
%\end{proof}

Using the properness of the $G$-action on $Y$,  it is easy to ensure in addition that the $\rho$-system $(\mu_x)_{x\in X}$  be  $G$-equivariant {{(see \cite{Williams}[Proposition 2.5])}}, i.e.
\begin{itemize}
\item ($G$-equivariance) \;\; For any $g\in G$ and $f\in C_c(Y)$, we have
$$
\int_{Y_{r(g)}} f(u)d\mu_{r(g)}(y) = \int_{Y_{s(g)}} f(y'g^{-1}) d\mu_{s(g)}(y').
$$
\end{itemize}
We shall also call such $\rho$-system an equivariant Haar system. Then the following statement is clear.

\begin{proposition}\label{EY}
The module $\maE_Y$ is a Hilbert $G$-module and the representation $\pi_Y: C_0(Y)\rightarrow \maL_{C_0(X)}(\maE_Y)$ given by multiplication operators is a $G$-equivariant non-degenerate representation.
\end{proposition}

\begin{proof}
That $\pi_Y$ is non-degenerate is clear. Recall the spaces
$$
Y\rtimes_s G:= \{ (y, g)\in Y\times G, \rho (y) = s(g)\}\text{ and }Y\rtimes_r G:= \{ (y, g)\in Y\times G, \rho (y) = r(g)\}.
$$
There is an isomorphism of Hilbert $C_0(G)$-modules between $s^*\maE_Y$ and the completion of $C_c(Y\rtimes_s G)$ with respect to the expected structures. The similar statement holds for $r^*\maE_Y$ and the completion of $C_c(Y\rtimes_r G)$. Thus, admitting these identifications, we can define  the unitary $V: s^*\maE_Y \rightarrow r^*\maE_Y$ which will automatically be a $C_0(G)$-linear map by setting for any
continuous compactly supported function $\eta$ on $Y\rtimes_s G$:
$$
(V\eta) (y, g) := \eta (yg, g)
$$
The pre-Hilbert $C_c(G)$-module structure  on $C_c(Y\rtimes_s G)$ (and similarly on $C_c(Y\rtimes_r G)$) is given  for  $\xi, \xi_1, \xi_2\in C_c(Y\rtimes_s G)$ and $ f\in C_c(G)$ by
$$
(\xi f) (y, g) := \xi (y, g) f (g) \text{ and } <\xi_1, \xi_2 >   (g):= \int_{Y_{s(g)}}\overline{\xi_1(y,g)}\xi_2(y,g) \; d\mu_{s(g)}(y).
$$
A straightforward verification then shows that the natural  map
$$
C_c(Y) \otimes C_c(G) \longrightarrow C_c(Y\times G),
$$
induces the above identifications of $s^*\maE_Y$ and $r^*\maE_Y$. We check using the $G$-invariance of the measures $(\mu_x)_{x\in X}$ that $(V^*\xi ) (y, g)= \xi(y g^{-1}, g)$ and hence that $V$  extends to a  unitary operator. Again by direct inspection, we obtain $\pi_{ij}^*s^*\maE_Y$ and $\pi_{ij}^*r^*\maE_Y$ by completing respectively with respect to the appropriate $C_0(G^{(2)})$-valued inner product the space of continuous compactly supported functions on respectively  the spaces
$$
Y\rtimes_{s, \pi_{ij}} G^{(2)} = \{(y, g_1, g_2)\in Y\times G^{(2)}, s\pi_{ij} (g_1, g_2)= \rho(y)\}
$$
and
$$
 Y\rtimes_{r, \pi_{ij}} G^{(2)} = \{(y, g_1, g_2)\in Y\times G^{(2)}, r\pi_{ij} (g_1, g_2)= \rho(y)\}.
$$
Hence we can write
$$
V_{01} \eta (y, g_1, g_2)=\eta (yg_1, g_1, g_2),  V_{12}\eta (y, g_1, g_2)=\eta (yg_2, g_1, g_2), V_{02} \eta (y, g_1, g_2)= \eta (yg_1g_2, g_1, g_2),
$$
which shows that $V$ satisfies the allowed relation for $\maE_Y$ to be a Hilbert $G$-module.

Recall on the other hand that the structure of $G$-algebra of $C_0(Y)$ is given by the similar $C_0(G)$-isomorphism $\alpha_Y: \varphi\mapsto [(y,g)\mapsto \varphi (yg, g)]$. Now
 for any $\varphi\in C_c(Y\rtimes_s G)$ and any $\xi\in C_c(Y\rtimes_r G)$:
$$
\left(V\circ (s^*\pi_Y)(\varphi) \circ V^*\right) (\xi) (y , g) = \varphi (yg, g) \xi (y, g) = (\alpha_Y\varphi)(y,g) \xi (y, g) = (r^*\pi_Y) (\alpha_Y(\varphi)) (\xi) (y,g),
$$
which shows that $\pi_Y$ is a $G$-equivariant representation.
\end{proof}

So we get using the specific Hilbert $G$-module $\maE_Y$ the dual $C^*$-algebra  $D^*_G(Y, \maE_Y)$  together with its closed two-sided involutive ideal  $C^*_G(Y, \maE_Y)$ and the quotient $C^*$-algebra $Q^*_G(Y, \maE_Y)$. {{On the other hand,  using the proper (and free) $G$-space $Y':=Y\rtimes_r G$, we also obtain the Hilbert $G$-module $\maE_{Y'}$  which corresponds to the field of Hilbert spaces over $X$ whose fiber at $x\in X$ is $L^2(Y_x, \mu_x)\otimes \ell^2(G^x)$, with the obvious extended representation $\pi_Y\otimes \id$ of $C_0(Y)$ in $\maE_{Y'}$. Hence, we also get the corresponding Roe $C^*$-algebras $D^*_G(Y, \maE_{Y'})$, $C^*_G(Y, \maE_{Y'})$ and $Q^*_G(Y, \maE_{Y'})$.}}
{{Notice indeed that the anchor map for the $G$-space  $Y'$ is $\rho'=\rho\circ p_1$ with $p_1:Y\rtimes_r G\to Y$ being the first projection. On the other hand, recall as well the groupoid structure on $Y'$ which is the crossed product structure, so with ${r} (y, g) =y$ and ${s} (y, g)= yg$ and with unit space $Y$. This structure is pulled back from $G$ and we have commutative diagrams where $\pi_2$ is the second projection ($\pi_2 (y,g)=g$)}}:

\[
\begin{CD}
Y\rtimes_r G @> {{s}, {r}}   >>Y \\
@V\pi_2 VV     @VV \rho V\\
G @> {s, r} >> X
\end{CD}
\]

We can now state {{(Compare \cite{Roe})}}:

\begin{theorem}\label{Morita}
Assume  that $Y$ is a proper $G$-space which is $G$-compact. Then the $C^*$-algebras $C^*_{\redg}(G)$ and $C^*_G(Y,\maE_{Y'})$ are Morita equivalent.
\end{theorem}

{{We shall more precisely identify the $C^*$-algebra $C^*_G(Y,\maE_{Y'})$ with the $C^*$-algebra of compact operators of some full Hilbert $C^*_{\redg} (G)$-module $L^2_G(Y')$. Let us first give some results which hold for the two Hilbert modules $\maE_Y$ and $\maE_{Y'}$. For simplicity, we give them for $Y$ and explain later on the needed modifications for $Y'$.}} The Hilbert module $L^2_G(Y)$ is the Connes-Skandalis Hilbert module  obviously extended to the non-smooth case, see \cite{ConnesSkandalis, Roe}.
The space $C_c(Y)$ of continuous compactly supported functions on $Y$ can indeed be endowed with the pre-Hilbert module structure over the convolution compactly supported algebra $C_c(G)$ of $G$. The rules are defined by:
\begin{itemize}
\item $
(\xi f)(y): = \sum_{\gamma\in G_{\rho(y)}} \xi(y\gamma^{-1})f(\gamma),  \text{ for }f\in C_c(G), \xi\in C_c(Y).
$
\item
$
<\eta,\xi>(g): = \int_{y\in Y_{r(g)}} {\overline{\eta(y)}} \xi(yg)  d\mu_{r(g)}(y), \text{ for }\eta,\xi\in C_c(Y), g\in G.
$
\end{itemize}
Passing to  completions, we obtain our Hilbert $C^*_{\redg}(G)$-module  $L^2_G(Y)$. The similar construction yields the Hilbert $C^*_{\redg}(G)$-module  $L^2_G(Y')$.

Recall on the other hand the regular representation $\lambda=  (\lambda_x)_{x\in X}$ for the  groupoid $G$. So  $\lambda_x : C^*_{\redg}(G)\rightarrow \maL(\ell^2(G_x))$ is given on the dense subalgebra $C_c(G)$ by:
$$
\lambda_x(f)(\xi)(g):= \sum_{g'\in G_x} f(g {g'}^{-1})\, \xi(g'), \quad \text{ for }f\in C_c(G), \xi\in C_c(G_x)\text{ and } g\in G_x.
$$
%\begin{remark}\ (M2I+M: NEEDS TO BE EXPANDED SOME MORE?)
%As a consequence, we see that for any $\phi\in C_0(Y)$, while  the $C_0(G)$-linear operator $r^* (\pi_Y\phi)$  on the Hilbert module $r^*\maE_Y$ is not necessarily in the range of the representation $r^*\pi_Y: C_0(Y\rtimes_r G) \rightarrow \maL_{C_0(G)} (r^*\maE_Y)$, it belongs to the range of the extended morphism to the multipliers $C_b(Y\rtimes_r G)$.  It is more precisely given by $(r^*\pi_Y)(\hat{r}^*\phi)$. Since ${\hat r}$ is also the first projection, this operator is also induced by $\pi_Y(\phi)\otimes \id: \maE_Y\otimes C_0(G)\rightarrow \maE_Y\otimes C_0(G)$ which is the multiplication operator by the bounded function on $Y\times G$ given by $\pi_1^*\phi = \phi\otimes 1_G$ with $1_G$ the constant function equal to $1$ on $G$.
%\end{remark}

The space of arrows $G^{(1)}$ is itself an interesting example of a proper right $G$-space $Y$  which is moreover a free $G$-space. The anchor map here is the source map $\rho_G=s$ which is surjective and open (since $G$ is \'etale) so that the action reduces to the composition on $G^{(2)}$. Moreover, the Haar system here corresponds to the counting measures on each $G_x$.
We get with this example the (full) Hilbert $C_0(X)$-module  $\maE_G$ when we simply specify $Y=G^{(1)}$. Then  $\maE_G$ is associated with the continuous field of Hilbert spaces $(\ell^2(G_x))_{x\in X}$. Notice though that here the fibers of the anchor map are discrete and that the space $G^{(1)}$ is $G$-compact only when $X$ is compact.

\begin{lemma}
The regular representation yields an (injective) $*$-homomorphism $\lambda : C^*_{\reg} G\rightarrow \maL_{C_0(X)}(\maE_G)$ which is valued in the space of $G$-invariant operators.
\end{lemma}

\begin{proof}
Denote by $V_G$ the unitary of $\maE_G$ which defines the $G$-action. Recall that $V_G$ is induced by the map
$$
V_G: C_c(G\rtimes_s G) \rightarrow C_c(G^{(2)}) \text{ given by } V_G \varphi (g_1, g_2) = \varphi (g_1g_2, g_2), \quad (g_1, g_2)\in G^{(2)}.
$$
Here $G\rtimes_s G= \{(g, g')\in G^2, s(g)= s(g')\}$. But for $k\in C_c(G)$, $\varphi\in C_c(G\rtimes_s G)$ and $(g_1,g_2)\in G^{(2)}$ we can write
\begin{eqnarray*}
\left(V_G\circ s^*\lambda (k)\right) (\varphi) (g_1, g_2) & = & s^*\lambda(k)(\varphi) (g_1g_2, g_2)\\
& = & \sum_{g\in G_{s(g_2)}} k(g_1g_2g^{-1}) \varphi (g, g_2)
\end{eqnarray*}
On the other hand,
\begin{eqnarray*}
\left(r^*\lambda (k)\circ V_G\right) (\varphi) (g_1, g_2) & = &  \sum_{g'\in G_{s(g_1)}} k(g_1 {g'}^{-1}) (V_G\varphi) (g', g_{{2}})\\
& = &  \sum_{g'\in G_{s(g_1)}} k(g_1 {g'}^{-1}) \varphi (g'g_2, g_{{2}}).
\end{eqnarray*}
Setting $g'g_2=g$ in this last expression gives the equality
$$
r^*\lambda (k)\circ V_G  = V_G\circ s^*\lambda (k)
$$
\end{proof}

The operator $\id\otimes_{\lambda} V_G$ denotes the well defined $C_0(X)$-adjointable operator which is a unitary from the Hilbert module $L^2_G(Y)\otimes_{s^*\lambda} s^*\maE_G$ to the Hilbert module $L^2_G(Y)\otimes_{r^*\lambda} r^*\maE_G$. Now, we have by definition
$$
s^*(L^2_G(Y)\otimes_{\lambda} \maE_G) \simeq L^2_G(Y)\otimes_{s^*\lambda} s^*\maE_G \text{ and } r^*(L^2_G(Y)\otimes_{\lambda} \maE_G)\simeq  L^2_G(Y)\otimes_{r^*\lambda} r^*\maE_G.
$$

\begin{proposition}\label{tensorisom}
We have an isometric $*$-isomorphism of $G$-Hilbert $C_0(X)$-modules:
$$
\Phi\; : \; L^2_G(Y)\otimes_{\lambda}\maE_G\stackrel{\cong}{\longrightarrow} \maE_Y.
$$
which is given  for $\zeta\in C_c(Y), \xi\in C_c(G)$, and $y\in Y$ by:
$$
\Phi (\zeta\otimes \xi)(y):=\sum_{g\in G_{\rho(y)}} \; \zeta(yg^{-1})\, \xi(g).
$$
{{In the same way, we have the similar isometric $*$-isomorphism of $G$-Hilbert $C_0(X)$-modules $
\Phi :L^2_G(Y')\otimes_{\lambda}\maE_G \cong \maE_{Y'}$ given by the similar formula.}}
\end{proposition}

\begin{proof}
%We shall use this time and for simplicity the identification with the fields of Hilbert spaces and show that we have for any $x\in X$, an isomorphism of Hilbert spaces:
%$$
%\Phi_x: L^2_G(Y)\otimes_{\lambda_x}\; \ell^2(G_x)\stackrel{\cong}{\longrightarrow} L^2(Y_x,\mu_x).
%$$
{{We only give the proof for $\maE_Y$ since the proof for $\maE_{Y'}$ is similar.}} For $f, \xi \in C_c(G)$ and $\zeta \in C_c(Y)$, we have by direct inspection $\Phi (\zeta f\otimes \xi)=\Phi (\zeta\otimes \lambda (f)\xi)$. Indeed, the two expressions give at $y\in Y$:
$$
\sum_{g, g'\in G_{\rho(y)} } \zeta (yg^{-1}) \, f(g{g'}^{-1})\, \xi (g'), \quad  \text{ for }y\in Y.
$$
%
%
%
%Computing the LHS,
%\begin{eqnarray*}
%\Phi(\zeta.f\otimes \xi)(y)&=& \sum_{\gamma\in G_x}(\zeta.f)(y\gamma^{-1})\xi(\gamma)\\
%&=& \sum_{\gamma\in G_x}\sum_{\alpha\in G_{r(\gamma)}}\zeta(y\gamma^{-1}\alpha^{-1})f(\alpha)\xi(\gamma)\\
%&=& \sum_{\gamma\in G_x}\sum_{\beta\in G_{x}}\zeta(y\beta^{-1})f(\beta\gamma^{-1})\xi(\gamma)\quad (\beta=\alpha\gamma)\\
%\end{eqnarray*}
%
%Now the RHS:
%\begin{eqnarray*}
%\Phi(\zeta\otimes \pi_x^{reg}(f)\xi)(y)&=& \sum_{\gamma_1\in G_x}\zeta(y\gamma_1^{-1})[\pi_x^{reg}(f)\xi](\gamma_1)\\
%&=& \sum_{\gamma_1\in G_x}\sum_{\gamma_2\in G_{x}}\zeta(y\gamma_1^{-1})f(\gamma_1\gamma_2^{-1})\xi(\gamma_2)\\
%\end{eqnarray*}
%
%Comparing the last lines of the two computations we see that they are equal.\\
It is also easy to check that $\Phi$ is an isometry.
%
%For the LHS, we have:
%
%\begin{eqnarray}
%\label{LHSiso}
%<\Phi(\zeta\otimes\xi),\Phi(\zeta\otimes \xi)>&=&\int_{Y_x} \Phi(\zeta\otimes \xi)(y)\overline{\Phi(\zeta\otimes \xi)(y)}  d\mu_x(y)\nonumber\\
%&=& \int_{Y_x} \left(\sum_{\gamma\in G_x} \xi(\gamma)\zeta(y\gamma^{-1})\right)\left(\overline{ \sum_{\alpha\in G_x}\xi(\alpha)\zeta(y\alpha^{-1}) } \right) d\mu_x(y)\nonumber\\
%&=& \sum_{\gamma\in G_x} \sum_{\alpha\in G_x} \xi(\gamma)\overline{\xi(\alpha)} \int_{Y_x} \zeta(y\gamma^{-1}) \overline{\zeta(y\alpha^{-1})} d\mu_x(y)
%\end{eqnarray}
%
%Now the RHS:
%\begin{eqnarray*}
%<\zeta\otimes\xi,\zeta\otimes \xi>&=& <\xi,\pi_x^{reg}(<\zeta,\zeta>)\xi>\\
%&=& \sum_{u\in G_x} \xi(u)\overline{[\pi_x^{reg}(<\zeta,\zeta>)\xi](u)}\\
%&=& \sum_{u\in G_x} \xi(u) \sum_{v\in G_x} \overline{\xi(v)<\zeta,\zeta>(uv^{-1})} \\
%\end{eqnarray*}
%
%Recall that
%$$
%<\zeta,\zeta>(w):= \int_{Y_{r(w)}} <\zeta(y),\zeta(yw)> d\mu_{r(w)}(y)
%$$
%
%So we have,
%\begin{eqnarray}
%\label{RHSiso}
%\sum_{u\in G_x} \xi(u) \sum_{v\in G_x} \overline{\xi(v)<\zeta,\zeta>(uv^{-1})} &=& \sum_{u\in G_x} \xi(u) \sum_{v\in G_x} \overline{\xi(v)}  \int_{Y_{r(u)}} <\zeta(y),\zeta(yuv^{-1})> d\mu_{r(u)}(y)\nonumber\\
%&=& \sum_{u\in G_x} \xi(u) \sum_{v\in G_x} \overline{\xi(v)}  \int_{Y_{x}} <\zeta(y'u^{-1}),\zeta(y'v^{-1})> d\mu_x(y')
%\end{eqnarray}
%
%Comparing the equations \ref{LHSiso} and \ref{RHSiso} now gives the claim.\\
%
We complete the proof by pointing out that the space $C_c(Y)$ is contained in the range of $\Phi$. Let $\zeta\in C_c(Y)$ be given and denote by $K$ the image under $\rho$ of the support of $\zeta$, a compact subspace of $X$. With our assumptions on $X$, we can find a continuous compactly supported function $\varphi$ on $X$ which is identically $1$ on $K$. Since the unit space $X=G^{(0)}$ is a clopen subspace of $G$, we deduce that $\varphi$ extends trivially to a continuous compactly supported function $\delta$ on $G$. It is then clear that $\Phi (\zeta\otimes \delta) = \zeta$.

It remains to show the $G$-equivariance of $\Phi$. Let $V$ be as {{in the proof of Proposition \ref{EY}}} the unitary which defines the $G$-action on the Hilbert $C_0(X)$-module $\maE_Y$, let $\xi\in C_c(Y)$ and let $k\in C_c(G\rtimes_s G)$ be given. Then we compute for $(y, g)\in Y\rtimes_r G$:
$$
(V\circ s^*\Phi) (\xi\otimes k) (y, g) = (s^*\Phi) (\xi\otimes k) (yg, g) = \sum_{g'_1\in G_{\rho (yg)}}  \xi (ygg_1^{-1}) k (g_1, g).
$$
while in the same way we obtain
$$
(r^*\Phi)\circ (\id \otimes _{\lambda} V_G) (\xi\otimes k) (y, g) = (r^*\Phi) (\xi\otimes V_G k) (y, g)= \sum_{g_2\in G_{\rho (y)}} \xi (yg_2^{-1}) k (g_2 g , g).
$$
Setting in this last expression $g_2g=g_1$, the proof is complete.
\end{proof}

%MOULAY STOPPED HERE 21.9.2015

The representations $s^*\lambda$ and $r^*\lambda$ used above are the  induced ones on $s^*\maE_G$ and $r^*\maE_G$ respectively by $\lambda\otimes \id$. Notice that we sometimes denote them simply by $\lambda$ when no confusion can occur.
The  isomorphism $\Phi$ defined above induces the allowed Morita equivalence  of Theorem \ref{Morita}. More precisely,

\begin{proposition}
\label{Cstariso}
The map $\Phi_*: \maL_{C^*_{\reg}(G)}(L^2_G(Y))\rightarrow \maL_{C_0(X)}(\maE_Y)$ given by $\Phi_*(T)=\Phi\circ(T\otimes _\lambda \id)\circ\Phi^{-1}$ induces a $C^*$-isomorphism $\maK_{C^*_{\redg}(G)}(L^2_G(Y))\cong C^*_G(Y,\maE_Y)$.\\
{{The same statement holds if we replace $\maE_Y$ by $\maE_{Y'}$. More precisely, we also have the $C^*$-isomorphism
$$
\Phi_*\; : \; \maK_{C^*_{\redg}(G)}(L^2_G(Y'))\stackrel{\cong}{\longrightarrow} C^*_G(Y,\maE_{Y'})
$$}}
\end{proposition}

\begin{proof}
Let $\eta_1,\eta_2\in C_c(Y)\subset L^2_G(Y)$. Recall the (Hilbert module) rank one operator $\theta_{\eta_1, \eta_2}$ on $L^2_G (Y)$ defined by
$$
\theta_{\eta_1, \eta_2}  (\xi) := \eta_1 <\eta_2, \xi>.
$$
Notice that $s^*(\theta_{\eta_1,\eta_2} \otimes_{\lambda} \id)= \theta_{\eta_1,\eta_2}\otimes_{s^*\lambda} \id$. Since $\Phi$ is $G$-equivariant we can write
$$
V\circ s^*\Phi = r^*\Phi \circ (\id\otimes_{\lambda} V_G) .
$$
But $(\id\otimes_{\lambda} V_G) \circ (\theta_{\eta_1,\eta_2}\otimes_{s^*\lambda} \id)=(\theta_{\eta_1,\eta_2}\otimes_{r^*\lambda} \id)\circ  (\id\otimes_{\lambda} V_G)$ hence
\begin{eqnarray*}
V\circ (s^*\Phi) \circ (\theta_{\eta_1,\eta_2}\otimes_{s^*\lambda} \id) \circ s^*\Phi^{-1}& = & (r^*\Phi ) \circ (\theta_{\eta_1,\eta_2}\otimes_{r^*\lambda} \id)\circ (\id\otimes_{\lambda} V_G) \circ s^*\Phi^{-1}\\
& = & (r^*\Phi ) \circ r^*((\theta_{\eta_1,\eta_2}\otimes_{\lambda} \id)) \circ r^*\Phi^{-1} \circ  V.
\end{eqnarray*}
Hence if we denote by  $K_{\eta_1, \eta_2}$ the function on $Y\times_\rho Y$ given by
$$
K_{\eta_1, \eta_2} (y, y') := \sum_{g\in G_{\rho (y)}} \eta_1(y g^{-1}) {\overline{\eta_2(y'g^{-1})}},
$$
then the operator $\Phi_*\theta_{\eta_1, \eta_2}$ is given by the expression
$$
\Phi_*\theta_{\eta_1, \eta_2} (\xi) (y) = \int_{\rho^{-1} (y)} K_{\eta_1, \eta_2} (y, y') \xi (y') d\mu_{\rho(y)} (y').
$$
 It is  clear then that $\Phi_*\theta_{\eta_1, \eta_2}$ also has finite propagation with respect to  $d_Y$. Moreover, for any continuous compactly supported function $\varphi$ on $Y$, the operators $\pi_Y(\varphi) \circ \Phi_*\theta_{\eta_1, \eta_2}$ and $\Phi_*\theta_{\eta_1, \eta_2}\circ \pi_Y(\varphi)$ are families of operators on $(L^2 (Y_x))_{x\in X}$ which are associated {{(through the fiberwise integral as above) with the continuous kernels
$$
^\varphi K_{\eta_1, \eta_2} (y, y') = \varphi (y) K_{\eta_1, \eta_2} (y, y') \text{ and } K^\varphi_{\eta_1, \eta_2} (y, y') =  K_{\eta_1, \eta_2} (y, y')\varphi (y').
$$
{{Notice that the crossed product groupoid $(Y\times_\rho Y) \rtimes_s G$ is an \'etale groupoid so that the counting measures on $G$ induce $(Y\times_\rho Y) \rtimes_s G$ with a continuous Haar system, this shows in turn that the kernels $^\varphi K_{\eta_1, \eta_2}$ and $K^\varphi_{\eta_1, \eta_2}$ are continuous (and compactly supported) on $Y\times_\rho Y$.  Indeed, the properness of the $G$-action on $Y$ shows that the continuous functions on $(Y\times_\rho Y) \rtimes_s G$ given by
$$
(y, y'; g) \longmapsto \varphi (y) \eta_1 (yg^{-1}) {\overline{\eta_2(y'g^{-1})}} \text{ and } (y, y'; g) \longmapsto \varphi (y') \eta_1 (yg^{-1}) {\overline{\eta_2(y'g^{-1})}}
$$
are compactly supported.}} A standard argument then shows that they are both compact operators on the Hilbert module $\maE_Y$.
Indeed, any  continuous compactly supported kernel as above on $Y\times_\rho Y$ can be uniformly approximated by linear combinations  of kernels from $C_c(Y)\otimes  C_c(Y)$ (elementary kernels) that we restrict to $Y\times_\rho Y$ and which can even be supposed to be supported within a fixed compact subset of $Y\times_\rho Y$. This can be seen for instance using first the Tietze theorem and then the usual approximation property. This allows to prove for instance that the associated operator can be approximated by {{finite rank operators of the Hilbert module $\maE_Y$}}. }}

%{Indeed, one can show that for a continuous compactly supported section $\xi=(\xi_x)_{x\in X}\in C_c(Y)$, the function
%$$
%x\mapsto <(\pi_Y(\varphi)\Phi_*\theta_{\eta_1,\eta_2})_x\xi_x,\xi_x>=\int_{Y_x}\int_{Y_x}\varphi(y)K_{\eta_1,\eta_2}(y,y')\xi_x(y')\overline{\xi_x(y)}
%d\mu_x(y) d\mu_x(y')
%$$
%is continuous with compact support in $X$ (see e.g. \cite{Paterson1}, Proposition 3.2), and at each $x\in X$ the operator $(\pi_Y(\varphi)\Phi_*\theta_{\eta_1, \eta_2})_x$ is a
%finite-rank operator on $L^2(Y_x,\mu_x)$.{I2M+I:This remark is not enough to show compactness on Hilbert module- probably have to use Arzela-Ascoli}}
Since $\Phi_*$ is continuous (actually an isometry) we deduce from the previous discussion that it sends the compact operators of the Hilbert module $L^2_G(Y)$ to operators in $C^*_G(Y,\maE_Y)$.

To finish the proof, we need to show that the operators $\Phi_*\theta_{\eta_1, \eta_2}$ span a dense subspace of $C^*_G(Y,\maE_Y)$. We  use  averaging for our proper groupoid $Y\rtimes G$ as follows, see \cite{Paterson2} which extends techniques from \cite{CoMo}[Section 1].

Given a {\underline{compactly supported}} $P\in \maL_{C_0(X)}(\maE_Y)$ we may consider its well defined average operator $\Av(P)\in \maL_{C_0(X)}(\maE_Y)$  given as
$$
\Av(P)_x=\sum_{g\in G^x} V_gP_{s(g)} V_g^{-1}.
$$
The sum is of course finite due to the $G$-properness of the space $Y$ and the compact support of $P$. The resulting operator $\Av(P)$ then has finite propagation. The proof that $\Av(P)$ is an adjointable operator (with adjoint $\Av (P^*)$) and hence belongs to $\maL_{C_0(X)}(\maE_Y)$ is classical, see for instance \cite{Paterson2}[Theorem 4]. Indeed the norm of $\Av (P)$ can be estimated using the norm of  $P$ but also its support.  Moreover, by construction, the operator $\Av (P)$ is $G$-invariant.

If now $T$ is an element of  $C^*_G(Y,\maE_Y)$ with finite propagation and $c$ is a compactly supported continuous cutoff function (recall that $Y$ is $G$-compact), then
the operator $\pi_Y(c) T$ belongs to $\maK_{C_0(X)} (\maE_Y)$ and has compact support contained in some space of the form $A\times_\rho A\subset Y\times_\rho Y$ with $A$ a compact subspace of $Y$. Moreover,  since $T$ is already $G$-invariant, we have the convenient relation:
\begin{equation}\label{Average}
T = \Av (\pi_Y(c) T).
\end{equation}
Fix $\ep >0$ and let $(\xi_i,\eta_i)_{i\in I}$ be a finite collection of elements of $\maE_Y$ (whose supports may be taken inside $A$) such that
$$
||\pi_Y(c)T- \sum_{i\in I} {{\widehat{\theta}}}_{\xi_i,\eta_i}||<\epsilon,
$$
{{where we have denoted for the sake of clarity by ${{\widehat{\theta}}}_{\xi_i,\eta_i}$  the rank one operator defined similarly to $\theta_{\eta_1, \eta_2}$ but now acting on the Hilbert module  $\maE_Y$.}} A density argument allows to further assume that the $\xi_i's$ and $\eta_i's$ live in $C_c(Y)$. Moreover the support of $\sum_{i\in I} {{\widehat{\theta}}}_{\xi_i,\eta_i}$ can be assumed as close as we please to that of $\pi_Y(c)T$. Therefore, we deduce the existence of a  constant $\kappa$ such that
$$
||\Av \left( \pi_Y(c)T)  - \sum_{i\in I} {{\widehat{\theta}}}_{\xi_i,\eta_i}  \right) ||  \leq \kappa || \pi_Y(c) T - \sum_{i\in I} {{\widehat{\theta}}}_{\xi_i,\eta_i}  ||.
$$
Since $\Av ( \pi_Y(c)T) = T$ and $\Av \left(\sum_{i\in I} {{\widehat{\theta}}}_{\xi_i,\eta_i}  \right) = \sum_{i \in I} \Phi_* {{\widehat{\theta}}}_{\xi_i, \eta_i}$, the proof is complete for the Hilbert module $\maE_Y$.

{{Now, all the above arguments hold as well for the Hilbert module $\maE_{Y'}$ with the extended representation $\pi_Y\otimes_{r^*} \id$, but  one has to use finite propagation of operators on $\maE_{Y'}$ according to our definition, say with respect to the representation of $C_0(Y)$. Formula \ref{Average} then still makes sense for a finite propagation operator $T$ on $\maE_{Y'}$ although the operator $(\pi_Y\otimes_{r^*} \id ) (c) T$ is nomore compactly supported but is only compactly supported with respect to the $Y$ variable. The rest of the proof is similar. }}

%
%
%
%
%
% This is shown below in the sequence of Lemmas (\ref{average})- (\ref{dense1}).
%
%
%Claims (TO DO):
%\begin{enumerate}
%\item such elements are dense in $C^*_G(\maE_Y)$.
%\end{enumerate}

\end{proof}

We have now completed the proof of our theorem.  More precisely:

\begin{proof} (of Theorem (\ref{Morita}))
{{Since $Y'$ is a free and proper $G$-space, the Hilbert module   $L^2_G(Y')$ is a full Hilbert $C^*_{\redg}(G)$-module, see \cite{MuhlyRenaultWilliams}[Proposition 2.10]. Therefore the assertion of the theorem follows immediately from Proposition \ref{Cstariso}}}.
\end{proof}

{{\begin{remark}\label{semi-ample}
It is  obvious from the above proof that Theorem \ref{Morita} holds with $\maE_Y$ instead of $\maE_{Y'}$ when the action of $G$ on $Y$ is assumed to  be free (and proper).
\end{remark}}}

%
%\begin{remark}\cite{GengouxTuXu}\
%%\cite{MoriyoshiNatsume,Paterson1,Paterson2}\label{Hilbertidentification}\
%There exists a unique topology  on the space $\amalg_{x\in X} \maK(L^2(Y_x, \mu_x))$   which endows it with the structure of a continuous field of $C^*$-algebras over $X$, and so that for any $\xi, \eta\in \maE_Y$, the rank one operator $\theta_{\xi, \eta}$ belongs to $C_0(X, \amalg_{x\in X} \maK(L^2(Y_x, \mu_x)))$, the space of continuous sections of this field. Moreover,  the  space of $C_0(X)$-compact operators on the Hilbert module $\maE_Y$ can then  be   described as the continuous sections $(T_x)_{x\in X}$ of this continuous field of compact operators.
%\end{remark}

\section{Compatibility with the Baum-Connes map}\label{CompPaschke}

  We use the notations of the previous sections, in particular the Hilbert $G$-module $\maE_Y$ is associated with the field of Hilbert spaces $L^2(Y_x, \mu_x)$. {{The compatibility theorems proved in the present section hold for $\maE_Y$ as well as  for $\maE_{Y'}$ and for simplicity we only give the proofs for the first Hilbert module and leave the easy modifications as an easy verification for the interested reader.}}
The  short exact sequence of $C^*$-algebras
$$
0\rightarrow C^*_G(Y,\maE_Y)\longrightarrow D^*_G(Y,\maE_Y)\longrightarrow Q^*_G(Y,\maE_Y)\rightarrow 0.
$$
together with Bott periodicity, yields  the following periodic  six-term exact sequence of $K$-groups
\begin{eqnarray}\label{surgeryseq}
\begin{CD}\label{SixTerm}
K_0(C^*_G(Y,\maE_Y)) @>>> K_0(D^*_G(Y,\maE_Y)) @>>> K_0(Q^*_G(Y,\maE_Y)) \\
@A{\del_1}AA   @.  @VV{\del_0}V\\
K_{1}(Q^*_G(Y,\maE_Y)) @ <<< K_1(D^*_G(Y,\maE_Y)) @<<< K_1(C^*_G(Y,\maE_Y))\\
\end{CD}
\end{eqnarray}
%that will be called the analytic surgery exact sequence for the groupoid $G$:

In this section we shall prove  the compatibility of the connecting maps $\del_i, i=0,1$ with the classical Baum-Connes  map, as described for instance in \cite{Tu}. See \cite{ConnesBook} for a more detailed description of this latter for \'etale groupoids and its relation with important conjectures in geometry and topology, especially in the study of foliations. Our result, Theorem \ref{compatibilityBC}  below, is well known for discrete groups \cite{Roe} and our method is an extension of Roe's proof to groupoids and Hilbert modules associated with groupoids.

\subsection{The Paschke-Higson map}

We define here the Paschke-Higson maps
$$
\maP_*:K_*(Q^*_G(Y, \maE_Y))\longrightarrow KK^{*+1}_G(Y,X), \quad *=0, 1.
$$
In  the even case, it is easy to see that  we can take a class $y\in K_0(Q^*_G(Y, \maE_Y))$ which is represented by a self-adjoint operator $P \in D^*_G(Y,\maE_Y)$ satisfying the relation
$$
P^2=P  \quad \text{ modulo } C^*_G(Y,\maE_Y).
$$
Recall that $C_0(Y)$ and $C_0(X)$ are $G$-algebras, and that $\maE_Y$ is a $G$-Hilbert $C_0(X)$-module such that $\pi_Y: C_0(Y)\rightarrow \maL(\maE_Y)$ is a $G$-equivariant representation. Also  the operator $2P-1$ is self-adjoint and satisfies $(2P-1)^2=\id$ up to $C^*_G(Y,\maE_Y)$,  moreover it  is \emph{exactly} $G$-invariant (not only up to compacts), hence the triple $(\pi_Y,\maE_Y, 2P-1)$ represents a class in $KK^{1}_G(Y,X)$.

We thus define the class  $\maP_0 (y)\in KK^{1}_G(Y,X)$ by setting
$$
\maP_0 (y) := \left[(\pi_Y, \maE_Y, 2P-1)\right].
$$
Using invariance of $KK$-classes under operator homotopy, the universal property of Grothendieck groups, and that  $KK$-classes don't see the operation of adding  degenerate cycles, it is easy to check that   $\maP_0 (y)$ is well defined, that is: it only depends on the class $y$ of $P$ in $K_0(Q^*_G(Y, \maE_Y))$.

To define similarly the Paschke-Higson map $\maP_1: K_1(Q^*_G(Y, \maE_Y))\longrightarrow KK^{0}_G(Y,X)$ corresponding to  the odd case, we let similarly $y\in K_1(Q^*_G(Y, \maE_Y))$ be a class which is  represented by an operator $u \in D^*_G(Y,\maE_Y)$ satisfying
$$
uu^*=I \text{ and }u^*u=I  \quad \text{ modulo } C^*_G(Y,\maE_Y)
$$
Then we set
$$
\maP_1(y):= \left[(\pi_Y\oplus \pi_Y, \maE_Y\oplus \maE_Y, \begin{bmatrix}   0 & u^*\\ u& 0 \end{bmatrix})\right]
$$
Again the triple $(\pi_Y\oplus \pi_Y, \maE_Y\oplus \maE_Y, \begin{bmatrix}   0 & u^*\\ u& 0 \end{bmatrix})$ is obviously a $\Z_2$-graded $G$-equivariant Kasparov cycle which represents a class in $KK_G^0 (Y, X)$. Moreover, the class $\maP_1(y)$ is well defined, i.e. only depends on the class $y$ of $u$ in $K_1(Q^*_G(Y, \maE_Y))$ and not on the representative $u$.

%M2(I+M): EXTEND THESE CONSTRUCTIONS TO CLASSIFYING SPACE!

%\subsection{The BC maps for $G$}

We now recall the Baum-Connes map associated with the proper $G$-compact space $Y$ \cite{Tu}.  So, associated with the proper $G$-compact space $(Y, \rho)$, there is a  (Baum-Connes) index map
$$
\mu_{BC, Y}^* : KK_G^* (Y, X) \longrightarrow K_{*} (C^*_{\redg} (G)),
$$
that we proceed to recall now for the convenience of the reader. See again \cite{Tu} and also \cite{HigsonNov}.

The map $\mu_{BC, Y}^*$ will be the composite map of two standard constructions that we call respectively  ''the descent map'' and ''the KM contraction'' in reference to Kasparov-Michschenko, i.e.
$$
\mu_{BC, Y}^*\; : \; KK^{*}_G(Y,X)\xrightarrow{descent} KK^*(C_0(Y)\rtimes_{\redg}G, C^*_{\redg}G)\xrightarrow{p_{KM}}KK^*(\C, C^*_{\redg}G) \cong K_{*}(C^*_{\redg}G)
$$
The KM contraction is given by reducing to the image of a Michschenko projection $p_{KM}\in C_0(Y)\rtimes_{\redg}G$ and was defined by Kasparov.
More precisely such projection defines a class $[p_{KM}]$ in $K_0(C_0(Y)\rtimes_{red}G)\simeq KK (\C, C_0(Y)\rtimes_{red}G)$ and the map $p_{KM}$ is given as the Kasparov product with this class.

On the other hand, the descent map was introduced by Kasparov for groups and extended by Le Gall to groupoids in \cite{LeGall}.
For a ($\Z_2$-graded) $G$-Hilbert module $E$, one can define the crossed-product Hilbert $C^*_{\redg}G$-module $E\rtimes G$, it is by definition given by an interior tensor product as follows:
$$
E\rtimes G:= E\otimes_{C_0(X), r^*} C^*_{\redg}G
$$
where the action of $C_0(X)$ on $C^*_{\redg}G$ is given via the pull-back map induced by the range map $r: G\rightarrow X$. Note that $E\rtimes G$ inherits a $\Z_2$-grading from $E$.
Any  $G$-equivariant, degree-preserving representation $\pi: C_0(Y)\rightarrow \maL_{C_0(X)}(E)$ induces the crossed-product representation
$$
\pi\rtimes \lambda: C_0(Y)\rtimes_{\redg} G \rightarrow \maL_{C^*_{\redg}G}(E\rtimes G)
$$
which is defined as follows. For $\xi\in C_c(Y), \phi \in C_c(G), \eta \in \pi(C_c(Y))E, \alpha\in C_c(G)$, and $g\in G$
$$
\pi\rtimes\lambda(\xi\otimes \phi)(\eta\otimes\alpha)(g):= \sum_{g'\in G^{r(g)}} [\pi(\xi)V^E_{g'}\eta].(\phi(g')\alpha(g'^{-1}g) )
$$
where $V^E \in \maL(s^*E, r^*E)$ is the unitary implementing the $G$-action on $E$.

Now suppose that $(\pi, E, F)$ is a triple representing a class in $KK^*_G(Y, X)$. Then the triple $(\pi\rtimes\lambda, E\rtimes G, F $ defines a $KK$-cycle in
$KK_*(C_0(Y)\rtimes_{red}G, C^*_{\redg}G)$ (see e.g. \cite{Tu} and \cite{Kasparov2}). We end up in this way with the Kasparov descent map
$$
j_G: KK^*_G(Y, X)\longrightarrow KK_*(C_0(Y)\rtimes_{red}G, C^*_{\redg}G) \text{ defined by }
j_G([(\pi,E,F)]):= [(\pi\rtimes\lambda, E\rtimes G, F ].
$$
Moreover, with $c$ being the cutoff function defined above for the cocompact proper $G$-action on $Y$, the element $e\in C_c(Y\rtimes_r G)\subseteq C_0(Y)\rtimes_{\redg} G$ given by
$$
e(y,g):= \sqrt{c}(y)\sqrt{c}(yg)
$$
is a projection in $C_0(Y)\rtimes_{\redg}G$. Therefore $e$ defines the Kasparov-Michschenko  class
$p_{KM}$ which is viewed as an element of  $KK_*(\C, C_0(Y)\rtimes_{\redg}G)$. Kasparov cup-product with this element gives a map:
$$
KK_*(C_0(Y)\rtimes_{\redg} G, C^*_{\redg} G)\xrightarrow{p_{KM}\otimes \bullet} KK_*(\C, C^*_{\redg} G)
$$
Composition of this map with the descent map $j_G$ is the Baum-Connes map associated with $Y$:
$$
\mu_{BC,Y}^*: KK^*_G(Y, X)\longrightarrow K_*(C^*_{\redg}G).
$$
Let us also recall the universal Baum-Connes map for complenetess \cite{Tu}. If $Y$ is a locally compact proper $G$-space which is not necessarily $G$-compact, then the Baum-Connes map for $Y$ is defined by an inductive limit over $G$-compact closed subspaces. More precisely, for any $G$-compact closed subspace $Y'$ of $Y$, we have the above map $\mu_{BC,Y'}^*$, and if $Y''\subset Y'$ is an inclusion of $G$-compact closed subspaces of $Y$, then the $G$-equivariant restriction  morphism $C_0(Y') \subset C_0(Y'')$ yields the map $KK^*_G(Y'', X) \to KK^*_G(Y', X)$ which can easily be seen to be compatible with the Baum-Connes maps $\mu_{BC,Y''}^*$ and $\mu_{BC,Y'}^*$. Hence, there is a well defined Baum-Connes map for the $G$-proper locally compact space $Y$ which is well defined on the inductive limit, over all $G$-compact closed subspaces, denoted
$$
RK^*_G (Y, X) \; := \;  \lim_{Y'\subset Y, Y'/G\; {\text{compact}}} \; KK^*_G(Y', X)
$$
In \cite{Tu}, a locally compact model for the classifying space of proper $G$-actions is constructed and we denote it $\underline {E}G$.

\begin{definition}
The universal Baum-Connes map for our \'etale groupoid $G$ is the well defined morphism
$$
\mu_{BC}^*: RK^*_G(\underline{E}G, X) \longrightarrow K_*(C^*_{\redg}G)
$$
\end{definition}

\begin{remark}
If $F$ is any additional $G$-algebra then we end up using the above construction with the Baum-Connes assembly map with coefficients in $F$:
$$
\mu_{BC, F}^*: RK^*_G(\underline{E}G, F) \longrightarrow K_*(F\rtimes_{\redg}G),
$$
so that $\mu_{BC}^*=\mu_{BC, C_0(X)}^*$.
\end{remark}

\subsection{The compatibility theorem}

We have proved in the previous section that the $C^*$-algebra {{$C^*_G(Y, \maE_{Y'})$}} is Morita equivalent to the reduced $C^*$-algebra $C^*_{\redg}  (G)$  associated with the \'etale groupoid $G$. {{This isomorphism result will be needed in the next papers of this series but will not be used here. Recall that if the action of $G$ on $Y$ is free for instance then $\maE_Y$ is always full and no need to use    the slightly modified Hilbert module $\maE_{Y'}$. We shall for simplicity rather give the constructions and proofs for  the  Hilbert module $\maE_Y$ and only point out that all the constructions can be easily modified so as to apply to the Hilbert module $\maE_{Y'}$.}}

Recall that we have constructed an explicit Hilbert $C^*_{\redg} (G)$-module $L^2_G(Y)$ whose compact operators are isomorphic through the map $\Phi_*$ to   $C^*_G(Y, \maE_Y)$. {{So, assuming that this module is full,   the $K$-theory isomorphism}}
$$
\maM_*: K_*(C^*_G(Y, \maE_Y)) \longrightarrow K_*(C^*_{\redg} (G)),
$$
induced by this Morita equivalence, can be described using Kasparov's $KK$-theory as the cup product with the class of the (trivially $\Z_2$-graded) even cycle  $(L^2_G(Y), \Phi_*, 0)$.

\begin{theorem}\label{compatibilityBC}
With the previous notations,  the following diagram commutes:
\[
\begin{CD}\label{BaumConnesassembly}
K_*(Q^*_G(Y,\maE_Y)) @> \del_*    >> K_{*+1}(C^*_G(Y,\maE_Y)) \\
@V\maP_* VV    @V\maM_{*+1}VV \\
KK^{*+1}_G(Y,X) @> \mu^{*+1}_{BC} >> K_{*+1}(C^*_{\redg}G)\\
\end{CD}
\]
where $\del_*$ is the connecting map in  (\ref{surgeryseq}), $\maP_*$ is the Pashcke-Higson map and $\mu^*_{BC}$ is  the Baum-Connes assembly map recalled in the previous paragraph.
\end{theorem}

The Paschke-Higson map is known to be an isomorphism for many classes of groupoids, especially for discrete countable groups, and also for groupoids associated with discrete countable group actions on spaces. This latter result is proved in the second paper of this series using an equivariant family version of the Voiculescu theorem. Therefore, Theorem \ref{compatibilityBC} relates the Baum-Connes conjecture for the groupoid $G$ with vanishing rigidity results.

\subsubsection{Proof of Theorem \ref{compatibilityBC} in the even case}
Our goal is thus to prove the commutativity of the following diagram:
\begin{equation}\label{BaumConnesassembly}
\begin{CD}
K_0(Q^*_G(Y,\maE_Y)) @> \del_0    >> K_1(C^*_G(Y,\maE_Y)) \\
@V\maP_0 VV     @VV \maM_1 V\\
KK^1_G(Y,X) @> \mu^1_{BC} >> K_1(C^*_rG)\\
\end{CD}
\end{equation}

As already observed, we can start with an operator $P\in D^*_G(Y,\maE_Y)$  which satisfies the relations $P=P^*$ and $P^2- P \in C^*_G(Y,\maE_Y)$, and represents a class $[P]\in K_0(Q^*_G (Y, \maE_Y))$. No need here to use matrix algebras which would yield to the same argument. Recall the Hilbert module $L^2_G(Y)$ over the reduced $C^*$-algebra $C^*_{\redg}(G)$ with the representation $\Phi_*$ of $C^*_G(Y, \maE_Y)$.
The computation of  $\maM\circ \del$ is given in the  following

\begin{proposition}\label{maM}
Denote by $\mathfrak{F}$ the adjointable operator on $L^2_G(Y)$ which corresponds to $2P-I$ through the isomorphism $\Phi_*$  of Proposition \ref{Cstariso}. Then the image of $[P]$ under the composite map $\maM_1\circ \del_0$ is represented by the odd Kasparov  $ (\C, C^*_{\redg} (G))$ cycle
$$
(L^2_G( Y) , \pi_\C, \mathfrak{F})
$$
where $\pi_\C$ is the trivial representation by multiplication of scalars.
\end{proposition}

%
%\begin{remark}\label{blackadar}
%As pointed out to us by the referee, and since the $C^*$-algebra $C^*_G(Y,\maE_Y)$ is $\sigma$-unital, Proposition \ref{maM} can be viewed as a corollary of Proposition 17.5.5 in \cite{Blackadar}. We though give  a direct proof below since some arguments in  \cite{Blackadar} are  only scketched for the experts. The same comment applies to the justification of Proposition \ref{OddCase} below.
%\end{remark}

\begin{proof}
{{Recall  that $P\in D^*_G(Y, \maE_Y)$, and descends to a self-adjoint idempotent in the quotient $C^*$-algebra $Q^*_G(Y,\maE_Y)$. In order to deduce the expected representative of the Kasparov class $\pa [P]$ in $KK_1(\C,C^*_G(Y,\maE_Y))$, we simply apply Proposition 17.5.5 in \cite{Blackadar}, as suggested to us by the referee. Indeed, let us denote by $\iota: D^*_G(Y, \maE_Y)\hookrightarrow M^s (C^*_G(Y, \maE_Y))$ the inclusion  in the multiplier $C^*$-algebra $M^s (C^*_G(Y, \maE_Y))$ (with its strict topology). See again \cite{Blackadar}. Then $\iota$ induces the inclusion, still denoted $\iota$, of $Q^*_G(Y,\maE_Y)$ in the quotient algebra $Q^s (C^*_G(Y, \maE_Y))$ given by $M^s (C^*_G(Y, \maE_Y))/C^*_G(Y, \maE_Y)$. We thus have the following commutative diagram of $C^*$-algebra exact sequences:}}

\begin{eqnarray}\label{ExactSequences}
{{\begin{CD}
0\to C^*_G(Y,\maE_Y) @>>> D^*_G(Y,\maE_Y) @>>> Q^*_G(Y,\maE_Y)\to 0 \\
@V{=}VV   @ VV{\iota} V @VV{\iota}V \\
0\to C^*_G(Y,\maE_Y) @>>> M^s (C^*_G(Y, \maE_Y)) @>>> Q^s (C^*_G(Y, \maE_Y))\to 0\\
\end{CD}}}
\end{eqnarray}

{{The boundary map for the first sequence can then be deduced from the boundary map for the second sequence, indeed, one has}}
$$
{{\del_0=\del\circ \iota_*\; : \; K_0(Q^*_G(Y,\maE_Y)) \stackrel{\iota_*}{\longrightarrow} K_0(Q^s (C^*_G(Y, \maE_Y)) \stackrel{\del}{\longrightarrow} K_1(C^*_G(Y,\maE_Y)),}}
$$
{{where $\del$ is the isomorphism described in \cite{Blackadar}. From this discussion we deduce that when viewed through the isomorphism $K_1(C^*_G(Y,\maE_Y))\simeq KK_1 (\C, C^*_G(Y,\maE_Y))$, the class $\del_0[P]$ is represented by the cycle}}
$$
{{\left(C^*_G(Y, \maE_Y),  F \right) \text{ where } F=2P-1. }}
$$
{{Notice that $F^2-I = 4 (P^2-P) \in C^*_G(Y, \maE_Y)$. }} Now the transport of $F$ through the identification $ C^*_G(Y, \maE_Y)\otimes_{ C^*_G(Y, \maE_Y)} L^2_G(Y)\simeq L^2_G(Y)$  is the operator $\mathfrak{F}_0$ defined, on the dense submodule generated by $\Phi_*^{-1}(T)(\zeta)$ with $T\in  C^*_G(Y, \maE_Y)$ and $\zeta\in L^2_G(Y)$, by
$$
\mathfrak{F}_0 \left(\Phi_*^{-1}(T)(\zeta)\right)  = \Phi_*^{-1} (F\circ T) (\zeta).
$$
{In order} to show that $\mathfrak{F}_0=\mathfrak{F}$, it suffices to use the isomorphism $L^2_G(Y)\otimes_{\lambda} \maE_G \simeq \maE_Y$ of Proposition \ref{tensorisom},  and to compute the resulting operator arising from $\mathfrak{F}_0\otimes_{\lambda} \id$ on $\maE_Y$.   Let then $S\in \maK_{C^*_{\redg}(G)} (L^2_G(Y))$ be such that $\Phi_*S= \Phi\circ (S\otimes_\lambda \id)\circ \Phi^{-1} = T$ and  let $\xi\in \maE_G$ be given. Then we obtain
$$
(\mathfrak{F}_0\otimes_\lambda \id) (S(\zeta)\otimes_\lambda \xi)  = ( \mathfrak{F}_0 \circ S)(\zeta) \otimes_\lambda \xi
$$
while
$$
(F\circ \Phi) (S\zeta\otimes_\lambda \xi)  =  [\Phi\circ (\mathfrak{F}_0\otimes_\lambda \id)] (S\zeta\otimes_\lambda \xi).
$$
Hence the proof of Proposition \ref{maM} is complete.

\end{proof}

\begin{remark}\
{{We have used in  the previous proof Proposition 17.5.5 in \cite{Blackadar} to identify $\partial [P]$. Since only the idea of the proof of that proposition is given  in \cite{Blackadar},  we point out that $\partial [P]$ is by definition represented in $K_1 (C^*_G(Y, \maE_Y))$ by the multiplier unitary $e^{2i\pi P}$, which obviously differs from $I$ by an element of $C^*_G(Y, \maE_Y)$ and also belongs to $D^*_G(Y, \maE_Y)$. Hence the identification of the corresponding Kasparov cycle in $KK_1(\C, C^*_G(Y, \maE_Y))$ is standard.}}
\end{remark}

%Now, we have an isomorphism $\Psi:(C^*_G(Y,\maE_Y)\oplus C^*_G(Y,\maE_Y))\otimes_{C^*_G(Y,\maE_Y)}L^2_GY\cong L^2_GY\oplus L^2_GY$ \cite{Lance}, and $Q$ acts on the Hilbert module on the right hand side as an adjointable operator denoted $\hat{Q}:= \Psi(Q\otimes Id)\Psi^{-1}$.

%Claim: There is an unitary isomorphism $V: \hat{Q}(L^2_GY\oplus L^2_GY)\rightarrow L^2_GY$ given by
%$$
%V(\xi, \exp(2i\pi P)\xi):= \xi, \quad V^*\xi:= ((\hat{F}\otimes Id)\xi, (\hat{F}\otimes Id)\exp(2i\pi P)\xi)
%$$
\medskip

We now proceed to compute $(\mu_{BC}^1\circ \maP_0)[P]$.
Let again  $Y\rtimes_r G$ be the groupoid induced by the $G$-action on $Y$, i.e. pulled back using $r$ and defined in Section \ref{metricGspace}. We construct the Hilbert $C^*_{\redg}(G)$-module $\maE_Y\rtimes G$ as usual, see for instance \cite{Tu}.
The inner product and module structure of this module  are described explicitly by
\begin{enumerate}
\item For $\varphi\in C_c(G)$ and  $\xi\in C_c(Y\rtimes_r G)$:
$$
(\xi \varphi)(y,g): = \sum_{g'\in G_{\rho(y)}} \xi(y, {g'}^{-1}) \varphi(g'g)
$$
\item For $\eta,\xi\in C_c(Y\rtimes_r G)$ and $g\in G$:
$$
<\eta_1,\eta_2>(g): = \sum_{g'\in G_{r(g)}}\int_{y\in Y_{r(g')}}  {\overline{\eta_1(y,g')}}\, \eta_2 (y, g' g) d\mu_{r(g')}(y).
$$
\end{enumerate}

The Hilbert $C^*_{\redg}(G)$-module $\maE_Y\rtimes G$ carries a representation $\pi_{Y\rtimes_r G}$ of $C_0(Y)\rtimes_{red} G\cong C^*_{\redg}(Y\rtimes_r G)$ given,
for $\phi\in C_c(Y\rtimes_r G)\subseteq C^*_{\redg}(Y\rtimes_r G)$ and $\xi \in C_c(Y\rtimes_r G)\subseteq \maE_Y\rtimes G$, by:
$$
\pi_{Y\rtimes_r G} (\phi)(\xi)(y,g):= \sum_{g'\in G^{r(g)}}    \phi(y,g')  \; \xi(yg', g'^{-1}g) =\sum_{g_1\in G_{s(g)}} \phi(y, yg_1^{-1})\; \xi (ygg_1^{-1}, g_1).
$$

Another interpretation of $\maE_Y\rtimes G$ is by considering the composition Hilbert module over $C^*_{\redg} (G)$ given by $\maE_Y\otimes_{C_0(X),r} C_{\redg}^*(G)$ where we view $C^*_{\redg} (G)$ as a Hilbert $C^*_{\redg} (G)$-module where $C_0(X)$ represents through multiplication with pull-backs by $r^*$. More precisely,
the map $\Psi: C_c(Y)\otimes_{C_0(X),r} C_c(G)\rightarrow C_c(Y\rtimes_r G)$ given by:
\begin{eqnarray}\label{iso1}
\Psi(\xi\otimes f)(y,g):= \xi(y)f(g), \quad \text{ for }\xi\in C_c(Y), f\in C_c(G),
\end{eqnarray}
can be easily seen to be an isometric isomorphism of Hilbert modules.

%
%\begin{remark}
%The above lemma shows that the definition of the Hilbert module $\maE_Y\rtimes G$ coincides with the one defined in \cite{Tu}: for a $G$-$C^*$-algebra $A$ and any Hilbert $A$-module $E$, a Hilbert $A\rtimes G$-module $E\rtimes G$ is defined as:
%$$
%E\rtimes G:= E\otimes_A A\rtimes G
%$$
%where $A$ acts on $A\rtimes G$ via the embedding $A\hookrightarrow M(A\rtimes G)$. Noticing that $C^*_{\redg}G$ can be described as $C_0(X)\rtimes G$ via the identification $X\rtimes G\cong G$, we get the compatibility of the our definition of $\maE_Y\rtimes G$ and that of \cite{Tu} given above.
%\end{remark}

Recall the cut-off function $c\in C_c(Y, [0, 1])$ from Definition (\ref{cutoff}).

\begin{proposition}\label{isometry}
There is an isometry of Hilbert $C^*_{\redg} (G)$-modules $I: L^2_G (Y)\rightarrow \maE_Y\rtimes G$ given by:
$$
I(\xi)(y,g):= \sqrt{c}(y)\xi(yg), \text{ for }\xi\in C_c(Y).
$$
The range of $I$ coincides with the Hilbert submodule which is the range of $\pi_{Y\rtimes_r G} (e)$ with $e$ the self-adjoint idempotent in $C_c(Y\rtimes_r G)$ given by $e(y, g) := \sqrt{c}(y)\sqrt{c}(yg)$.
\end{proposition}

\begin{proof}
This is a straightforward computation. It is easy to see that $I$ is adjointable and that $I^*$ is given for $\xi\in C_c(Y)$ by
$$
I^*(\eta) (y) = \sum_{g\in G_{\rho(y)}} \sqrt{c}(yg^{-1}) \, \eta (yg^{-1}, g).
$$
Therefore we get by direct inspection
$$
I^*\circ I = \id \text{ and } (I\circ I^*) (\eta) (y, g) = \sqrt{c}(y) \sum_{g'\in G_{s(g)}} \sqrt{c}(yg {g'}^{-1}) \, \eta (yg{g'}^{-1}, g'), \text{ for }\eta\in C_c(Y\rtimes_r G).
$$
This last formula is precisely $\pi_{Y\rtimes_r G} (e) (\eta) (y, g)$.
\end{proof}

Let as before $F=2P-I \, \in D^*_G(Y,\maE_Y)\subseteq \maL_{C_0(X)} (\maE_Y)$.
Then $F$ yields the operator $F\otimes \id_{C^*_{\redg}G} \in \maL(\maE_Y\otimes_{C_0(X),r}C^*_{\redg}G)$ and hence using the isomorphism from (\ref{iso1})
an operator $\bar{F}$ on the Hilbert module $\maE_Y\rtimes G$. Moreover, and as explained above, there is a well defined operator $\mathfrak{F}$  on $L^2_G (Y)$
which is associated with  $F$ through the isomorphism $\Phi_*$ of Proposition (\ref{Cstariso}). Since $F$ is $G$-invariant, it is easy to see that the operator $\mathfrak{F}$ coincides with the operator  $F$ on the the common dense domain $C_c(Y)$. {{In order}} to complete  the proof of the commutativity of Diagram \eqref{BaumConnesassembly},
we prove the following

\begin{lemma}\label{compactperturb}
We have $
I^*\circ \bar{F}\circ I-  \mathfrak{F}\; \in \; \maK_{C^*_{\redg}(G)}\, \left(L^2_G(Y)\right)
$
\end{lemma}

\begin{proof}\
We may assume in this proof that $F$ has finite propagation. Indeed since $\maK_{C^*_{\redg}(G)}(L^2_G(Y))$ is closed and since the proof only uses the fact that $F\in D^*_G(Y,\maE_Y)$, an easy density argument then allows to conclude. Recall from Proposition \ref{Cstariso} that the map $\Phi_*$ is an isometric isomorphism from $\maK_{C^*_{\redg}(G)} (L^2_G(Y))$ to the $C^*$-algebra $C^*_G (Y, \maE_Y)$.
So, we shall prove that the operator
$$
\Phi_*[I^*\circ \bar{F}\circ I]- F\; = \; \Phi\circ \left([I^*\circ \bar{F}\circ I]\otimes_\lambda \id_{\maE_G}\right) \circ \Phi^{-1} -  F \; \in \; C^*_G (Y, \maE_Y).
$$
But if we fix $\xi\in C_c(Y) \subset L^2_G(Y)$ and $\varphi\in C_c (G) \subset \maE_G$ and we set $\eta:=\Phi (\xi\otimes \varphi)$ then we may write
\begin{eqnarray*}
\Phi_*[I^*\circ \bar{F}\circ I] (\eta) (y) & = & \sum_{g\in G_{\rho (y)}} (I^*\circ {\bar F}\circ I) (\xi) (yg^{-1}) \varphi (g)\\
& = & \sum_{g\in G_{\rho (y)}}  \varphi (g) \sum_{k\in G_{r(g)}} {\sqrt c} (yg^{-1} k^{-1}) {\bar F} (I\xi) (yg^{-1} k^{-1}, k).
\end{eqnarray*}
But identifying the operator $F$ with the corresponding field $(F_x)_{x\in X}$, we can write
\begin{eqnarray*}
{\bar F} (I\xi) (yh^{-1}, h) & = & F\left(z\mapsto (I\xi) (z, h) \right) (yh^{-1})\\
& = & F ({\sqrt c} (h\xi)) (yh^{-1})\\
& = & [F, \pi_Y({\sqrt c})] (h\xi) (yh^{-1}) + (\pi_Y({\sqrt c}) \circ F) (h\xi) (yh^{-1})\\
& = & [F, \pi_Y({\sqrt c})] (h\xi) (yh^{-1}) + {\sqrt c} (yh^{-1}) F(\xi) (y).
\end{eqnarray*}
This is more precisely and for a fixed $h\in G_{\rho (y)}$ given by
$$
[F, \pi_Y({\sqrt c})]_{r(h)} ((h\xi)_{r(h)}) (yh^{-1}) + {\sqrt c} (yh^{-1}) F_{\rho(y)} (\xi_{\rho(y)}) (y),
$$
where only the restriction of $\xi$ to $Y_{\rho(y)}$ is involved.
We deduce from this computation
\begin{multline*}
(I^*\circ {\bar F}\circ I) (\xi) (y)=\sum_{h\in G_{\rho (y)}} {\sqrt c} (yh^{-1})\left([F, \pi_Y({\sqrt c})] (h\xi) (yh^{-1})  + {\sqrt c} (yh^{-1}) F(\xi) (y)  \right)\\ = F_{\rho (y)}(\xi)(y) + \sum_{h\in G_{\rho (y)}} {\sqrt c} (yh^{-1}) [F, \pi_Y({\sqrt c})]_{r(h)} (h\xi) (yh^{-1}).
\end{multline*}
Therefore,
\begin{multline*}
\Phi_*[I^*\circ \bar{F}\circ I] (\eta) (y)  =  \sum_{g\in G_{\rho (y)}} \varphi (g) \sum_{k\in G_{r(g)}} {\sqrt c} (yg^{-1} k^{-1}) \\ \left([F, \pi_Y({\sqrt c})] (k\xi) (yg^{-1}k^{-1}) + {\sqrt c} (yg^{-1} k^{-1}) F(\xi) (yg^{-1}) \right)
\end{multline*}
Hence we get
$$
\Phi_*[I^*\circ \bar{F}\circ I] (\eta) (y)  =   \Phi (\mathfrak{F}\xi\otimes \varphi) (y) +   \sum_{g\in G_{\rho (y)}} \varphi (g) \sum_{k\in G_{r(g)}} {\sqrt c} (yg^{-1} k^{-1})  [F, \pi_Y({\sqrt c})]_{r(g)} (k\xi) (yg^{-1}k^{-1}).
$$
Now, $ \Phi (\mathfrak{F}\xi\otimes \varphi)$ is nothing but $F (\eta)$ by definition of $\mathfrak{F}$. On the other hand, we define out of our operator $F$ the operator $T$ on $C_c(Y) \subset \maE_Y$ by setting for our $\xi\in C_c(Y)$:
$$
T (\xi) (y) := \sum_{k\in G_{\rho(y)}} {\sqrt c} (yk^{-1}) [F, \pi_Y({\sqrt c})]_{r(k)} (k\xi) (yk^{-1}) =\sum_{k\in G_{\rho (y)}} \left\{ k^{-1} \left(\pi_Y({\sqrt c})\circ [F, \pi_Y({\sqrt c})]\right)\right\}_{\rho(y)} (\xi) (y).
$$
Then we check that  $T$ is well defined since ${\sqrt c}$ is compactly supported, and it is obviously $C_c(X)$-linear. Also it extends to a $C_0(X)$-adjointable operator on $\maE_Y$ by direct inspection, in fact it is self-adjoint since $F$ and $\pi_Y({\sqrt c})$ are self-adjoint. Moreover, since the commutator $[F, \pi_Y({\sqrt c})]$ is compact on the Hilbert module $\maE_Y$, so is $T$.
Indeed, if $\eta_1, \eta_2$ are elements of $C_c(Y)\subset \maE_Y$ then
$$
\sum_{k\in G_{\rho(y)}} {\sqrt c} (yk^{-1}) \theta_{\eta_1, \eta_2} (k\xi) (yk^{-1}) = \left(\sum_{k\in G_{\rho(y)}} \pi_{\rho (y)} (k^{-1} {\sqrt c}) \circ (\theta_{k^{-1}\eta_1, k^{-1}\eta_2})_{\rho (y)} \right) (\xi) (y).
$$
Now notice that for any compactly supported function $\zeta$ on $Y$, the set $K_{\zeta,c}$  defined by
$$
K_{\zeta,c}:=\{ g\in G| \text{ there exists } y\in \supp(c) \text{ such that } yg\in \supp(\zeta)\},
$$
 is compact in $G$,  since the $G$-action on $Y$ is proper. So,  we see that the above sum is finite and hence we get a compact operator. One then deduces that $T$ is compact and also by easy verification that this operator $T$ has finite propagation if $F$ does and is in general an element of $C^*(Y, \maE_Y)$. We thus get for any $x\in X$
$$
[\Phi_*[I^*\circ \bar{F}\circ I] (\eta)  -  F (\eta)] _{x}= \sum_{g\in G_{x}}  T (g^{-1}\xi) \varphi (g)= T_x (\Phi (\xi\otimes _\lambda \varphi) )= T_x (\eta).
$$
\end{proof}

\begin{corollary}
The $KK$-cycles $(I(L^2_G (Y)), \lambda_\C,\pi_{Y\rtimes_r G}(e)\bar{F}\pi_{Y\rtimes_r G}(e))$ and $(L^2_G (Y), \lambda_\C, \maF)$ are equivalent.
\end{corollary}

\begin{proof}
Recall that $\pi_{Y\rtimes_r G}(e)=I\,I^*$ and that, by Lemma \ref{isometry}, $I$ is an isometry which identifies $L^2_G(Y)$ with the range $I(L^2_G (Y))$ in $\maE_Y\rtimes G$. Therefore, using the unitary isomorphism
$$
I^* : I(L^2_G (Y)) \longrightarrow L^2_G(Y),
$$
we deduce that the $KK$-cycle $(I(L^2_G (Y)), \lambda_\C, \pi_{Y\rtimes_r G}(e)\bar{F}\pi_{Y\rtimes_r G}(e))$ is conjugate to the $KK$-cycle $(L^2_G(Y), \lambda_\C, I^* {\bar F} I)$. Now,
applying  Lemma (\ref{compactperturb}), we see that $I^* \bar{F}I$ is a compact perturbation of $\mathfrak{F}$. This implies the assertion.
\end{proof}

Thus we have proved:

\begin{corollary}
The image of $[P]$ under $\mu_{BC}^1\circ\maP_0$ coincides with that under $\maM_1\circ \del_0$.
\end{corollary}

\subsubsection{Proof of Theorem \ref{compatibilityBC} in the odd case}

The similar result in the odd case can be deduced using the space $Y\times \R$ with the anchor map $\rho\circ p_1$ with $p_1: Y\times \R\to Y$ being the first projection. Then the groupoid $G$ needs to be replaced by the groupoid $G\times \Z$. This needs though some extra-arguments and since the direct proof is shorter, we chose it here. So, we now prove by a  direct computation the analogous result in the odd case, i.e.  the commutativity of the following diagram:

\[
\begin{CD}\label{BaumConnesassembly}
K_1(Q^*_G(Y,\maE_Y)) @> \del_1    >> K_0(C^*_G(Y,\maE_Y)) \\
@V\maP_1 VV     @VV \maM_0 V\\
KK^0_G(Y,X) @> \mu^0_{BC} >> K_0(C^*_rG)\\
\end{CD}
\]

\begin{proposition}\label{OddCase}
Given an element $u\in D^*_G(Y, \maE_Y)$ whose projection in $Q^*_G(Y, \maE_Y)$ is a unitary operator, the image under the map $\maM_0\circ \del_1$, of the class of $u$ in $K_1(Q^*_G(Y, \maE_Y))$, is represented  in $KK (\C, C^*_{\redg}(G))$, by the Kasparov $KK$-cycle
$$
\left( L^2_G(Y), L^2_G(Y), \left(\begin{array}{cc}  0 & \mathfrak{U}^*\\ \mathfrak{U} & 0\end{array}\right)\right),
$$
where $\mathfrak{U}$ is the adjointable operator on $L^2_G(Y)$ which corresponds to $u$ under the isomorphism $\Phi_*$ of Proposition (\ref{Cstariso}).
\end{proposition}

\begin{proof}
We know that $uu^*-1, u^*u-1\in C^*_G(Y,\maE_Y)$ by assumption. There exists a unitary (actually in the connected component of the identity) $U\in M_2(D^*_G(Y,\maE_Y))$ such that
\[
 \pi(U)=\begin{bmatrix} \pi(u) & 0\\ 0 & \pi(u^*) \end{bmatrix}\quad  \in U_2(Q^*_G(Y,\maE_Y))
\]

Denote by $P_0$ the projection $\begin{bmatrix} I & 0\\ 0 & 0 \end{bmatrix}$ and define the projection $P= UP_0 U^*$. Then the class $[P]-[P_0]\in K_0(C^*_G(Y,\maE_Y))$ is by definition $\del_1 ([u]$. The corresponding class in $KK^0(\C,C^*_G(Y,\maE_Y))$ is thus given by the Kasparov cycle $(\maE:=\maE^+\oplus \maE^-, F)$, where
$$
\maE^+:= \Im (P)\subseteq C^*_G(Y,\maE_Y)\oplus C^*_G(Y,\maE_Y), \, \maE^-:= C^*_G(Y,\maE_Y) \text{ and }F=\begin{bmatrix} 0 & P\circ i_1\\ \pi_1\circ P& 0 \end{bmatrix} \in \maL_{C^*_G(Y, \maE_Y)} (\maE).
$$
The operator $\pi_1\in  \maL_{C^*_G(Y,\maE_Y)}(C^*_G(Y,\maE_Y)\oplus C^*_G(Y,\maE_Y),C^*_G(Y,\maE_Y))$ is the projection onto the first factor and the operator $i_1: C^*_G(Y,\maE_Y)\rightarrow C^*_G(Y,\maE_Y)\oplus C^*_G(Y,\maE_Y)$ is the inclusion; we note that $\pi_1^*=i_1$. It is clear from the definition that $F$ is self-adjoint.

Let $U= \begin{bmatrix} u & w\\ v & u^* \end{bmatrix}$, since $\pi(U)=diag(\pi(u),\pi(u^*))$, we get $v,w\in C^*_G(Y,\maE_Y)$.
%Let us check first that $F^*=F$. We have:
%$$
%P\circ i_1=\begin{bmatrix} u & w\\ v & u^* \end{bmatrix}\begin{bmatrix} 1 & 0\\ 0 & 0 \end{bmatrix}\begin{bmatrix} u^* & v^*\\ w^*& u \end{bmatrix}\begin{bmatrix} 1 \\ 0  \end{bmatrix} = \begin{bmatrix} uu^* \\ vu^*  \end{bmatrix}
%$$
%
%While we have,
%$$
%\pi_1\circ P=\begin{bmatrix} 1 & 0 \end{bmatrix}\begin{bmatrix} u & w\\ v & u^* \end{bmatrix}\begin{bmatrix} 1 & 0\\ 0 & 0 \end{bmatrix}\begin{bmatrix} u^* & v^*\\ w^*& u \end{bmatrix} = \begin{bmatrix} uu^* & uv^* \end{bmatrix}
%$$
%
%Thus $(P\circ i_1)^*=\pi_1\circ P$ and consequently $F=F^*$.
%{I2M: changes made below}
Then
$$
F^2-I= \begin{bmatrix} P\circ i_1\circ \pi_1\circ P-P & 0\\ 0 & \pi_1\circ P\circ i_1-I \end{bmatrix}
$$

Since $\pi_1\circ P\circ i_1=uu^*$, by the hypothesis on $u$, we get $\pi_1\circ P\circ i_1-I\in C^*_G(Y,\maE_Y)$. On the other hand, the term $P\circ i_1\circ \pi_1\circ P-I $ is given by
$$
P\circ i_1\circ \pi_1\circ P-P= \begin{bmatrix} uu^*(uu^*-I) & (uu^*-I)uv^*\\ vu^*(uu^*-I) & v(u^*u - I)v^* \end{bmatrix}
$$
Since by hypothesis, the elements $uu^*-I$ and $u^*u-I$ belong to $C^*_G(Y,\maE_Y)$, it is clear that all the entries in the matrix above belong to $C^*_G(Y,\maE_Y)$.

So we have proved that $F^2-I\in C^*_G(Y,\maE_Y)$ and hence is a compact operator on the Hilbert module $\maE$.
%We claim that there.
Now consider the operator $U^*: C^*_G(Y,\maE_Y)\oplus C^*_G(Y,\maE_Y)\rightarrow C^*_G(Y,\maE_Y)\oplus C^*_G(Y,\maE_Y)$, then the map $\Psi$  given by
$$
\Psi=\begin{bmatrix} \pi_1\circ U^* & 0\\ 0 & I \end{bmatrix}: \maE^+\oplus \maE^-\rightarrow C^*_G(Y,\maE_Y)\oplus C^*_G(Y,\maE_Y)
$$
 is a unitary isomorphism $\Psi: \maE^+\oplus \maE^-\longrightarrow C^*_G(Y,\maE_Y)\oplus C^*_G(Y,\maE_Y)$. This is an easy  consequence of the relation $U^*P=P_0U^*$ and the verification is omitted.
Therefore, the Kasparov cycle  $(\maE^+\oplus \maE^-, F)$ is unitarily equivalent to $(C^*_G(Y,\maE_Y)\oplus C^*_G(Y,\maE_Y), \Psi\circ F\circ \Psi^{-1})$. Computing the operator $\hat{F}:= \Psi\circ F\circ \Psi^{-1}$, we get that it is given by
$$
\hat{F}= \begin{bmatrix}  0 & \pi_1 U^* P i_1\\ \pi_1PUi_1 & 0 \end{bmatrix} =\begin{bmatrix}  0 & u^*\\ u & 0 \end{bmatrix}
$$
Taking now the Kasparov product with the Morita cycle $M= [L^2_G(Y), 0]\in KK_0(C^*_G(Y,\maE_Y), C^*_rG)$ that we have already described, with the representation given by the inclusion of $C^*_G(Y,\maE_Y)$ in $\maL_{C^*_{\redg}(G)}(L^2_G(Y))$ through the Morita isomorphism $C^*_G(Y,\maE_Y)\cong\maK_{C^*_{\redg}(G)}(L^2_G(Y))$, we obtain the following class in $KK_0(\C,C^*_{\redg}(G))$:
$$
\left(L^2_GY\oplus L^2_GY, \lambda_\C, \begin{bmatrix}  0 & \mathfrak{U}^*\\ \mathfrak{U} & 0 \end{bmatrix}\right),
$$
where $\mathfrak{U}=u\otimes_{C^*_G(Y, \maE_Y)} \id_{L^2_G(Y)}$. But the representation of $D^*_G(Y, \maE_Y)$ in $L^2_G(Y)$ is precisely given by the isomorphism $\Phi_*^{-1}$.
\end{proof}

\begin{proposition}
The image of the class of $u$ under the composite map $\mu_{G}^0\circ \maP_1$  is represented by the Kasparov cycle
$$
\left(L^2_GY\oplus L^2_GY, \lambda_\C, \begin{bmatrix}  0 & \mathfrak{U}^*\\ \mathfrak{U} & 0 \end{bmatrix}\right).
$$
\end{proposition}

\begin{proof}
{In order} to compute the image of the class of $u$ under the composite map $\mu_{G}^0\circ \maP_1$, we apply the same construction as for the even case. Notice first that the image under the Paschke map $\maP_1$ of $u$ is easily described by the even $G$-equivariant Kasparov cycle
$$
\left(\maE_Y\oplus \maE_Y, \pi_Y\oplus\pi_Y, T=\begin{bmatrix} 0 & u^*\\ u & 0 \end{bmatrix}\right)
$$
and we need to represent the image of this latter cycle under the Baum-Connes map $\mu_G^0$. The computation is similar to the even case and we get that this image is represented by the Kasparov cycle
$$
\left(   \pi_{Y\rtimes G} (e) (\maE_Y\rtimes G) \oplus  \pi_{Y\rtimes G} (e) (\maE_Y\rtimes G),   \; \pi_{Y\rtimes G} (e)\circ {\bar T}\circ \pi_{Y\rtimes G} (e)\right),
$$
where ${\bar T}$ is the adjointable operator on $\maE_Y\rtimes G$ corresponding to $T\otimes_{C_0(X), r} C^*_{\redg} (G)$ under the isomorphism describe previously, and $e$ is the Michschenko idempotent also described previously. Recall that the isometry $I$ identifies $L^2_G(Y)$ with the orthocomplemented Hilbert submodule $ \pi_{Y\rtimes G} (e) (\maE_Y\rtimes G)$ of $\maE_Y\rtimes G$ and it satisfies more precisely that $II^* = \pi_{Y\rtimes G} (e)$. Therefore, the previous Kasparov cycle is unitarily equivalent to the Kasparov cycle
$$
\left( L^2_G(Y) \oplus L^2_G(Y), I^* {\bar T} I \right).
$$
If now $\mathfrak{T}$ is the $C^*_{\redg} (G)$-adjointable operator in $L^2_G(Y)$ corresponding to $T$ under the isomorphism $\Phi_*$, then we need to show that
$$
I^* {\bar T} I - \mathfrak{T} \quad \in \maK_{C^*_{\redg}(G)} (L^2_G(Y)).
$$
This is a corollary of  Lemma \ref{compactperturb}.
\end{proof}

\section{The coarse $G$-index for continuous $G$-families}

We explain in this last section how to define the coarse $G$-index of a $G$-equivariant family of Dirac-type operators on a $G$-proper continuous
family of bounded geometry smooth riemannian manifolds. Our construction relies on some classical results due to Shubin \cite{Shubin} and extends
the work of Paterson to the coarse category. In particular, the construction given is a generalization of the index class for laminations defined by Moore-Schochet
\cite{MooreSchochet} to the setting of (uniform) bounded geometry laminations.

\subsection{Continuous families of manifolds}\label{ContFamilies}

The material in this overview subsection is  taken from \cite{Paterson2} so we shall be brief.

%- Define $C^{\infty,0}$ manifolds, maps, vector bundles, $G$-actions, Riemannian metric fields,...

%\subsubsection{Continuous family of smooth manifolds}
\begin{definition}[$C^{\infty,0}$ maps]
Let $M,N$ be smooth manifolds and let $X$ be a locally compact Hausdorff space. A function $f: M\times X\rightarrow N\times X$ is said to be
of class $C^{\infty,0}$ if $f(M\times \{x\})\subset N\times \{x\}$ for each
$x\in X$, and the map $X\ni x\mapsto f (\bullet, x)\in C^\infty(M,N)$ is continuous.
\end{definition}

In the previous definition, $C^\infty(M,N)$ is given the
usual Fr\'{e}chet topology of uniform convergence on compact subsets together with derivatives of all orders.
Let $Y$ be an second countable locally compact Hausdorff space and $\rho: Y\rightarrow X$ be an open surjective map.

\begin{definition}[Continuous family of smooth manifolds]\label{ContSmooth}
The triple $(Y,\rho, X)$ is called a continuous family of smooth manifolds if $\exists  k\in \N$ and a collection $\{U_\alpha,\phi_\alpha\}_{\alpha\in A}$ where:
\begin{enumerate}
 \item $U_\alpha \subset Y$ is open, and $Y=\bigcup_{\alpha\in A} U_\alpha$,
 \item for each $\alpha \in A$, $\phi_\alpha: U_\alpha\rightarrow \rho(U_\alpha) \times V_\alpha$ is a fiber-preserving homeomorphism with $V_\alpha\subseteq \R^k$ an open subset.
 \item for any $\alpha,\beta\in A$ such that $U_\alpha\cap U_\beta\neq \emptyset$, the maps
 $\phi_{\alpha\beta}:=\phi_\beta\circ \phi_\alpha^{-1}: \phi_\alpha(U_\alpha\cap U_\beta)\rightarrow \phi_\beta(U_\alpha\cap U_\beta)$ is a $C^{\infty,0}$ function from
 $\rho(U_\alpha\cap U_\beta) \times (V_\alpha \cap V_\beta)$ to itself.
\end{enumerate}

The pairs $(U_\alpha,\phi_\alpha), \alpha \in A$ will be called local charts. We shall call $Y$ a $C^{\infty,0}$-manifold.
\end{definition}

In \cite{Paterson2}, the integer $k$ is assumed to be $\geq 1$ but we prefer here to include $k=0$ which corresponds to $\R^k$ being $\{\star\}$.
The notion of $C^{\infty,0}$-maps between continuous families is defined similarly as well as $C^{\infty,0}$-diffeomorphism for instance.
It is also standard to define fibrations of continuous families of manifolds as well as $C^{\infty,0}$-vector bundles or hermitian $C^{\infty,0}$-vector bundles over
continuous families of smooth manifolds, see again \cite{Paterson2}.  If for instance,  $(Y,\rho, X)$ be a continuous family of smooth manifolds,
then the space $TY = \bigcup_{x\in X} TY_x$ (as well as  all the functorially associated bundles) inherits the structure of $C^{\infty,0}$-vector bundle over $Y$.

\begin{lemma}[\cite{Paterson2}]
 Let $U=\{U_\alpha\}_{\alpha\in A}$ be a locally finite open cover of $Y$ consisting of local charts. Then there exists a partition of unity $\{\psi_\alpha\}_{\alpha\in A}$ consisting
 of $C^{\infty,0}$ functions subordinate to $U$.
\end{lemma}

Let $G$ be our \'{e}tale Hausdorff groupoid as before with $G^{(0)}=X$ then $G$ is in fact a continuous family groupoid in the sense of \cite{Paterson2} since we have included the case $k=0$ in Definition \ref{ContSmooth}. Recall the definition of a $G$-space (Definition (\ref{Gspace})).

\begin{definition}[Smooth $G$-spaces]
 Let $Y$ be a $C^{\infty,0}$-manifold which is endowed with a $G$-action, thus making it a $G$-space. The $G$-action is said to be of class $C^{\infty,0}$ if $\rho:Y\rightarrow X$ is a continuous family of smooth manifolds, and the structure map $\lambda: Y{\rtimes}_r G\rightarrow Y$ is of class $C^{\infty,0}$.
 \end{definition}

Notice that the space $Y\rtimes_r G$ carries a natural $C^{\infty,0}$ structure so that the previous definition  makes sense.
The previous definition makes sense for any continuous family groupoid $G$ \cite{Paterson2}, but the general case is not needed in the present paper.

\subsection{The coarse $G$-index}

We  assume from now on that the proper $G$-space $(Y, \rho:Y\to X)$ is a continuous family of smooth riemannian manifolds (also denoted $C^{\infty, 0}$ according to Connes' notation for laminations) such that the action is of class $C^{\infty, 0}$, see  \cite{Paterson2} or the short overview given in \ref{ContFamilies}. We are mainly interested in the coarse index for {\em{complete laminations}} and the Paterson formalism will be convenient for us.

We assume moreover that $Y$ has (uniformly over $X$) bounded geometry in the fiber direction, and we also assume that the $G$-action is fiberwise isometric. In particular,  we assume  that the injectivity radius associated with the induced Riemannian metric on each $Y_x=\rho^{-1} (x)$ is bounded below independently of $x$, so that we have well defined barycentric fiberwise coordinates with  $C^{\infty,0}$-bounded changes over $Y$, and also that the curvature tensor defined on each smooth fiber  of $\rho$ is $C^{\infty, 0}$-bounded over $Y$. In particular, the smooth manifolds $Y_x$ are all complete riemannian manifolds and we use the complete riemannian $G$-invariant distance associated with the $C^{\infty, 0}$ $G$-invariant riemannian metric
to see $Y$ as a $G$-family of proper metric spaces.  Finally, notice that there exist  $C^{\infty, 0}$-bounded partitions of unity which are subordinate to  covers by geodesic balls, see  \cite{Shubin, Paterson2}. All the geometric structure that we shall use in this section are assumed to have uniformly  bounded geometry in an obvious sense which extends the classical definitions of \cite{Shubin} to the setting of continuous families. We point out that for $m\geq 1$, any $C^{\infty, m}$-submersion $\rho$ satisfies the assumptions of a $C^{\infty, 0}$-family of smooth manifolds and Definition \ref{ContSmooth} extends the notion of a submersion to encompass the setting where the topological space $X$ is not even $1$-differentiable.

\begin{definition}[Bounded Propagation Speed]
A family $\D=(D_x)_{x\in X}$ of symmetric, first-order  elliptic differential operators on $\{L^2(Y_x, E_x)\}_{x\in X}$ such that the quantity
$$
C(\D) :=  \sup\{ ||\sigma(y,\xi)||_{\End(E_x)} : y \in Y, \xi \in S^*_yY_{\rho (y)} \}
$$
is finite, is said to be of (uniformly) bounded propagation speed.
\end{definition}

By elliptic we of course mean here fully elliptic in the sense of Shubin, see \cite{Shubin}. Using the results of \cite{Shubin} and \cite{Paterson2}, it is a routine argument to show that, under our assumptions, families of Dirac-type operators associated with  (uniformly) $C^{\infty, 0}$-bounded structures, do have (uniformly) bounded propagation speed. Observe that a family $\D$ of symmetric, first-order elliptic differential operators induces a family of self-adjoint, closed unbounded operators which we also denote by $\D$, cf. \cite{HRbook}, Chapter 10. Suppose that this induced family $\D$ is continuous (see Definition \ref{Continuousclosed}) and has bounded propagation speed. We denote the induced regular operator on the associated Hilbert module $\maE_{Y, E}$ by $\maD$. By the functional calculus of regular operators, we know that for any continuous complex valued bounded function $f$ on $\R$, the operator $f(\maD)$ is a well defined adjoiontable operator on the Hilbert module $\maE_{Y, E}$ associated with the continuous field of Hilbert spaces $(L^2(Y_x, E_x))_{x\in X}$. In particular, the wave operator $\exp(is\maD)$ can be defined in this way as  an adjointable operator on $\maE_{Y, E}$.

\begin{lemma}[Finite propagation property of the wave operator]
Let $f, g \in C_0(Y)$ such that $\supp(f)\cap \supp(g)=\emptyset$ and either $f$ or $g$ has compact support. Then there exists $\epsilon>0$ such that $\pi_Y(f)\exp(is\maD)\pi_Y(g)=0$ for $|s|<\epsilon$.
\end{lemma}

\begin{proof}
By the compatibility of functional calculi (\ref{compatibility}), the adjointable operator $\exp(is\maD)$ corresponds to the continuous family of bounded operators $\{\exp(isD_x)\}_{x\in X}$. Then, by \cite{HRbook}, Corollary 10.3.3 for finite propagation of wave operators on complete manifolds, we have
$$
\pi_x(f|_{Y_x})\exp(isD_x)\pi_x(g|_{Y_x})=0 \quad \text{ for all $s$ such that} |s|<1/c_x
$$
where $c_x$ is the propagation speed of $D_x$ and $\pi_x$ is the representation of $C_0(Y_x)$ on $L^2(Y_x,E_x)$ by pointwise multiplication.

Therefore for $|s|<1/C(\D)$, the continuous family of operators $\pi_x(f|_{Y_x})\exp(isD_x)\pi_x(g|_{Y_x})$ is identically zero. Since $\pi_Y(f)\exp(is\maD)\pi_Y(g)$ corresponds to this family, the proof is concluded.
\end{proof}

%Moulay stopped here. 14.12.2017\\

\begin{lemma}[Inverse Fourier representation]\label{invFourier}
Let $\phi\in \maS(\R)$ be a Schwartz function such that it has compactly supported Fourier transform. Then, for any $u, v\in C^{\infty,0}_c(Y)$ the function $s \mapsto <\exp(is\maD)u,v>$ is continuous, and  the following relation holds:
$$
<\phi(\maD)u,v>= \int_{\R} \hat{\phi}(s)<\exp(is\maD)u,v> ds
$$
\end{lemma}

\begin{proof}
This is a classical result which can be deduced from the spectral theorem. The reader can also consult the arguments in \cite{HankePapeSchick}[Lemma 3.6] which immediately carries over to our situation.
\end{proof}

\begin{definition}[Normalizing function]
 A function $\chi \in C(\R)$ is called normalizing if $\chi$ is an odd function such that $\chi(x)>0$ for $x>0$ and $\chi(x)\xrightarrow{x\rightarrow \infty} 1$.
\end{definition}

\begin{remark}\label{invFourier1}
The space of Fourier transforms of smooth compactly supported functions on $\R$, a subspace of the Schwartz space $\maS(\R)$ is dense in $C_0(\R)$. It is easy to see that the inverse Fourier representation of Lemma (\ref{invFourier}) also allows to define the operators $f(\maD)$ for any  $f\in C_0(\R)$. Moreover, using the extension of the Fourier transform to distributions, one can also define $\chi (\maD)$ for any normalizing function $\chi$ by using again the inverse Fourier representation of Lemma \ref{invFourier}.
\end{remark}

\begin{theorem}[Roe's Lemma]\label{Roe}
Let $\D$ be a $G$-equivariant family of first-order elliptic differential operators with (uniformly) bounded propagation speed, inducing a continuous field of closed, self-adjoint operators, also denoted $\D$. Denote the regular operator on $\maE_{Y,E}$ induced by $\D$ as $\maD$. Then for any normalizing function $\chi$, the operator $\chi(\maD)$ belongs to the Roe algebra $D^*_G(Y;\maE_{Y,E})$. \\Moreover, for  any $\phi \in C_0(\R)$, the operator $\phi(\maD)$ belongs to the ideal  $C^*_G(Y;\maE_{Y,E})$. In particular, the class in the quotient $C^*$-algebra $C^*_G(Y;\maE_{Y,E})$ defined by the operaror $\chi (\maD)$ is independent of the choice of $\chi$.
\end{theorem}

\begin{proof}\
It is clear that functional calculus preserves $G$-equivariance, so $\chi(\maD)\in \maL_{C_0(X)}(\maE_{Y,E})^G$. Let $\chi'$ be a normalizing function such that $\supp(\hat{\chi'})$ is included in $[-\epsilon, \epsilon]$ for $\epsilon>0$ small enough. Then, $\chi-\chi' \in C_0(\R)$ so that corresponding operator $\chi(\maD)-\chi'(\maD)$ is locally compact according to the {{second}} item. Therefore it suffices to show that $\chi'(\maD)\in D^*_G(Y;\maE_{Y,E})$. To show that $[\chi'(\maD), \pi_Y(f)]\in \maK_{C_0(X)}(\maE_Y)$ for all $f\in C_0(Y)$, from Lemma \ref{KasparovLemma} it suffices to show that $\pi_Y(f)\chi'(\maD)\pi_Y(g)$ is compact whenever $f,g \in C_0(Y)$ have disjoint supports and at least one of them have compact support. However, to this end we apply the finite propagation property of the wave operator in conjunction with the inverse Fourier representation of $\chi'(\maD)$ (cf. Lemma \ref{invFourier1}) as
$$
\chi'(\maD) = \int_{-\epsilon}^\epsilon \what{\chi'}(s)\exp(is \maD) ds
$$

Since $\pi_Y(f)\exp(is\maD)\pi_Y(g)=0$ for $|s|<\epsilon$ for small enough $\epsilon$, the operator $\pi_Y(f)\chi'(\maD)\pi_Y(g)$ is in fact zero (thus compact). Lastly, that $\chi(\maD)$ is a norm-limit of finite propagation operators can be shown following the techniques in \cite{HankePapeSchick}, Lemma 3.6.

% For any $\phi\in C_0(\R)$, a straightforward computation shows that $\pi_Y(f)\phi(\maD)$ is an adjointable operator from $\maE_{Y,E}$ to $\maE_{K,E}^{(1)}(\maD)$ for any $f\in C_c(Y)$, so that the composition
% $$
% \maE_{Y,E}\xrightarrow{\pi_Y(f)\phi(\maD)}\maE_{K,E}^{(1)}(\maD)\hookrightarrow \maE_{Y,E}
% $$
%  is compact due to the Rellich Lemma (\ref{Rellich}), where $K= \supp(f)$.
%It is enough to show the claim for the functions $\phi_1(x)= (x+i)^{-1}$ and $\phi_2(x)=(x-i)^{-1}$.
% Using the fact that a strong-* continuous family $t$ is norm-continuous if and only if $t^*t$
%is norm-continuous, it is then enough to show that for $\phi=\phi_1\phi_2$, the operator $\phi(\maD)=(1+\maD^2)^{-1}$ is  a uniform limit of finite propagation, locally-compact {{$G$-invariant operators}}.

For the second part of the theorem, it suffices to give the proof for $\phi$ in the subspace $J\subset \maS(\R)$ of Fourier transforms of smooth compactly supported functions. Now, for any $\psi \in J$, the operator $\psi(\maD)$
has finite propagation because of the integral representation of Lemma \ref{invFourier}. By density of $C^{\infty, 0}_c(Y)$ in $C_0(Y)$, it moreover suffices to show that  for any $f\in C^{\infty, 0}_c(Y)$ and any $\psi\in J$, $\psi (\maD)\pi_Y(f) \in \maK_{C_0(X)}(\maE_{Y, E})$.
From Proposition 10.5.2 of \cite{HRbook}, we know that $\psi (D_x)\pi_Y(f_{|_{Y_x}})$ is compact for each $x\in \rho(\Supp (f))$ and only the regularity with respect to the $X$-variable needs to be investigated. Now, a classical argument using convolution reduces the  proof to the case where $\hat\psi$ is supported within a small enough neighborhood of $0$. By the uniform speed propagation, we can thus assume that the finite propagation operator $\psi (\maD)$ is actually supported within a uniform in $X$ small enough neighborhood of the diagonals of the fibers. Again, assuming that $f$ is also supported within a small enough subset of $Y$, we are reduced to the case of an open trivializing set $U_\alpha$ as in Definition \ref{ContSmooth}, with a fiber-preserving homeomorphism
$$
\phi_\alpha: U_\alpha\rightarrow \rho (U_\alpha)\times V_\alpha\text{ with } V_\alpha\text{ an open set in }{{ \R^k}}.
$$
Now, the proof is reduced to this local situation and we need to show that the  family $(\psi (D_x)\pi_Y(f_{|_{Y_x}}))_{x\in X}$ is then continuous in $x$ for the operator norm.  {{But this latter is justified  in \cite{Paterson2}[Proposition 11] using classical results due to M. Shubin \cite{Shubin}[Section 6].}} We have now shown that the operator $\psi (\maD)\pi_Y(f)$ is a compact operator on the Hilbert $C_0(X)$-module, i.e. $\psi (\maD)\pi_Y(f) \in \maK_{C_0(X)}(\maE_{Y, E})$, and hence that $\psi (\maD)\in C^*_G(Y;\maE_{Y,E})$ for any $\psi\in J$ and finally also for any $\psi\in C_0(\R)$ by density.

We end this proof by pointing out that if $\chi_1$ and $\chi_2$ are two normalizing functions then the difference $\chi_1-\chi_2$ belongs to $C_0(\R)$ and hence $\chi_1(\maD) - \chi_2 (\maD)\in C^*_G(Y;\maE_{Y,E})$. Therefore, the two operators $\chi_1(\maD)$ and  $\chi_2 (\maD)$ yield the same class in the quotient Roe algebra $Q^*_G(Y;\maE_{Y,E})$.
%{I2M: Could you please do it?}
\end{proof}

We recall now the so-called Kasparov's lemma that was used in the above proof. Let $\pi:C_0(Y)\rightarrow \maL_{C_0(X)}(\maE)$ be a non-generate representation in the Hilbert $C_0(X)$-module $\maE$. Recall that $T \in \maL_{C_0(X)}(\maE)$  is pseudolocal if $[T,\pi_Y(f)] \in \maK_{C_0(X)}(\maE)$ for all $f\in C_0(Y)$. The following is a straightforward generalization of Kasparov's Lemma characterizing pseudolocal operators. The proof is a straightforward adaptation of \cite{HRbook}[Lemma 5.4.7] and is omitted.

\begin{lemma}[Kasparov's Lemma]\label{KasparovLemma}
An operator $T\in \maL_{C_0(X)}(\maE)$ is pseudolocal if and only if for all $f,g \in C_0(Y)$ such that $\supp(f) \cap \supp(g) = \emptyset$ and at least one of the functions $f$ and $g$ have compact support, the operator $\pi (f)T\pi (g) \in \maK_{C_0(X)}(\maE)$.
\end{lemma}

From Theorem (\ref{Roe}) we conclude that for a normalizing function $\chi$, the operator $\chi(\maD)\in D^*_G(Y;\maE_{Y,E})$, while $\chi^2(\maD)-1\in C^*_G(Y, \maE_{Y, E})$. Thus, $\frac{1}{2}(1+\chi(\maD))$ gives a well-defined element in the $K$-theory group $K_0(Q^*_G(Y; \maE_{Y,E}))$. Let $\del_0$ be the boundary map in the $K$-theory
long exact sequence induced by the short exact sequence
$$
0\rightarrow C^*_G(Y;\maE_{Y, E})\rightarrow D^*_G(Y;\maE_{Y, E})\rightarrow Q^*_G(Y; E_{Y, E})\rightarrow 0
$$
We can now give the following definition of the coarse index.

\begin{definition}[Coarse $G$-index]
The (odd) coarse $G$-index $\Ind_G(\maD)$ i defined as the class $\Ind_G(\maD)$ which is the image of the class $\left[\frac{1}{2}(1+\chi(\maD))\right]$ under the boundary map
$$
\del_0: K_0(Q^*_G(Y; \maE_{Y,E})) \longrightarrow K_1 (C^*_G(Y; \maE_{Y,E})).
$$
So,
 $$
 \Ind_G(\maD):= \del_0\left(\left[\frac{1}{2}(1+\chi(\maD))\right] \right)\in K_1(C^*_G(Y;\maE_{Y, E}))
 $$
\end{definition}

{{Since the boundary map $\del_0$ is an exponential map, it is easy to see that the index class is represented by the unitary associated with $\chi$ by the formula $-e^{i\pi\chi (\maD)}$. In particular, the index class could as well be a priori defined by the Cayley transform $(i+\maD)(i-\maD)^{-1}$. }}
We have  treated the odd case which corresponds for families of Dirac-type operators to assuming that the fibers $Y_x=\rho^{-1}(x)$ be odd dimensional. The even case is similar and uses the notion of Voiculescu $G$-isometry.
It is worthpointing out that when $Y$ is  $G$-compact, Morita equivalence allows to deduce a well defined  index class  in {{$K_* (C^*_r (G))$}}, and we recover in this way the index class defined by Paterson in \cite{Paterson2}. Notice that Paterson's index class  extends the classical construction of Moore-Schochet for laminations of compact spaces \cite{MooreSchochet}, which in turn extends the Connes-index class for smooth foliations on compact manifolds   \cite{ConnesIntegration}.  We end this paper with the following corollary.

\begin{corollary}
Assume that $\maD$ is invertible with a gap in its spectrum around $0$, then $\maD$ has a well defined transgressed coarse $G$-index class $[\maD]$ which lives in $K_0 (D^*_G(Y;\maE_{Y, E}))$ whose image under the functoriality map $K_0 (D^*_G(Y;\maE_{Y, E}))\rightarrow K_0(Q^*_G(Y; E_{Y, E}))$ is the class $\left[\frac{1}{2}(1+\chi(\maD))\right]$. \end{corollary}
\begin{proof}\
Indeed,  we may then  choose $\chi$ such that $\chi^2=1$ on the spectrum of $\maD$ and then we see that $\frac{1}{2}(1+\chi(\maD))$ is already a projection in $D^*_G(Y;\maE_{Y, E})$ which defines the class $[\maD]$ and which embodies the vanishing of the index class $\Ind_G(\maD)$.
\end{proof}

The previous proposition applies when the $G$-equivariant family of smooth bounded geometry manifolds is a spin family and admits (uniformly in the $X$-variable) positive scalar curvature. For foliations this is related with the recent results of \cite{BenameurHeitsch}. A similar statement happens when we only assume positive scalar curvature at infinity but using relative Roe algebras. These developments will be treated in a forthcoming paper in relation with the Gromov-Lawson relative index theorem.

%\{Question: How to define the even index- this requires a Voiculescu isometry!}
%

\bigskip

\appendix

\section{$G$-representations in $G$-modules}\label{AppendixA}

%\subsection{$G$-algebras}

We  give the definition and gather the properties of $G$-algebras and representations of $G$-algebras in $G$-modules, that will be used in the sequel. For most of the material about groupoid actions on $C^*$-algebras, we refer to  \cite{LeGall, Tu}. Given a  $C^*$-algebra $A$, we denote by $MA$ the  $C^*$-algebra of multipliers of $A$ \cite{Kasparov}, and the center of $M(A)$ is denoted by $ZM(A)$.

\begin{definition}\cite{LeGall}[Definition 3.1]
\begin{enumerate}
\item A $C_0(X)$-algebra is a couple $(A,\theta)$, where $A$ is a $C^*$-algebra and $\theta: C_0(X)\rightarrow ZM(A)$ is a $*$-homomorphism such that $\theta(C_0(X))A=A$.
%M2I: ACCORDING TO KASPAROV, ONE SHOULD TAKE THE CLOSURE. CAN YOU CHECK THE LITERATURE PLEASE.
\item A morphism $\phi: (A, \theta_A)\rightarrow (B, \theta_B)$ of $C_0(X)${-}algebras is a $*$-homomorphism $\phi: A\rightarrow B$ such that the following relation holds:
$$
\phi(\theta_A(f)a)=\theta_B(f)(\phi(a)), \quad \text{ for }a\in A \text{ and } f\in C_0(X).
$$
\end{enumerate}
\end{definition}

So, a $C_0(X)$-algebra structure on the $C^*$-algebra $A$ endows it with the structure of a $C_0(X)$-module.
Given a $C_0(X)$-algebra $A$, the fibre of $A$ at  $x\in X$ is $A_x:= A\, /\, \theta(C_x)A$, where $C_x$ is the ideal of functions in $C_0(X)$ vanishing at $x$. This yields an u.s.c.  field $(A_x)_{x\in X}$ of $C^*$-algebras.

The $C^*$-algebra $C_0(X)$ is itself a $C_0(X)$-algebra ($\theta$ is just the multiplication operator by the given function), and more generally for  any locally compact Hausdorff space $Y$ with a continuous map $\rho: Y\to X$, we may endow the $C^*$-algebra $C_0(Y)$ with the $C_0(X)$-algebra structure  corresponding to the $*$-homomorphism $\theta=\rho^*$ defined by $\rho^*(f)=f\circ \rho$.
\begin{remark}
If $f:Y\to X$ is a continuous map between locally compact Hausdorff spaces, then the u.s.c. field associated with the $C_0(X)$-algebra $(C_0(Y), f^*)$ corresponds to a continuous field in the sense of Dixmier if and only if $f$ is open. See for instance \cite{Blanchard, GengouxTuXu}.
\end{remark}
If we are given a $C_0(X)$-algebra $(A, \theta)$, then we can pull it back to a $C_0(Y)$-algebra $(\rho^*A, \rho^*\theta)$.
% Indeed, one sets:
%$$
%\rho^*A := C_0(Y)\otimes_{C_0(X)} A \text{ and } \rho^*\theta is given for } h, g \in C_0(Y) \text{ and }a\in A \text{ by } \; \rho^*\theta  (h) (g\otimes a) = hg\otimes a.
%$$
%Here we use the identifiaction
%$$
%g\rho^*(f) \otimes a \sim g\otimes \theta (f) a,\quad \text{for }g\in C_0(Y), f\in C_0(X)\text{  and }a\in A.
%$$
Moreover, given a morphism $\phi: (A, \theta_A)\rightarrow (B, \theta_B)$ of $C_0(X)${{-}}algebras, we easily get a morphism  of $C_0(Y)$-algebras between the pull-backs
$$
\rho^*\phi :(\rho^*A, \rho^*\theta_A)\longrightarrow  (\rho^*B, \rho^*\theta_B).
$$
See  again \cite{LeGall} for the details of these standard constructions.

In particular, we may consider the $C^*$-algebras $r^*A$ and $s^*A$ which are endowed with the structures of $C_0(G)$-algebras.  Recall the  set  of pairs of composable arrows   $G^{(2)}=\{ (g , g')\in G^2, s(g)=r(g')\}$ for our groupoid $G$.  Then we have the three maps $\pi_{01},\pi_{12}$ and $\pi_{02}$ from $G^{(2)}\rightarrow G^{(1)}$ corresponding respectively, to  projection onto the first component, projection onto the second component, and composition of arrows.
The following relations hold:
$$
r\circ \pi_{01}=r\circ \pi_{02},\quad s\circ\pi_{12}=s\circ \pi_{02},\quad \text{ and } \quad s\circ\pi_{01}=r\circ\pi_{12}.
$$

%Notice that they can as well be defined directly. Let for instance $I_r$ be the closed two-sided ideal in $C_0(X\times G)$ of functions that vanish on the closed graph of  $r$. Then the $C_0(G)$-algebra $r^*A$ is easily seen to coincide with the quotient
%$$
%r^*A:= \left(A\otimes C_0(G)\right)/I_r^A,
%$$
%and similarly for the $C_0(G)$-algebra $s^*A$.
%Indeed,  if we denote by  $i:C_0(G)\hookrightarrow C_b(G)$  the inclusion, then the morphism $\theta\otimes i$ descends to well defined $*$-homomorphisms
%$$
%r^*\theta : C_0(X\times G)/I_r \rightarrow ZM (r^*A) \text{ and } s^*\theta : C_0(X\times G)/I_s \rightarrow ZM (s^*A)
%$$
%which define the $C_0(G)$-algebra structures of $r^*A$ and $s^*A$.
%Notice also that  we have used  the obvious identifications
%$$
%C_0(X\times G)/I_r \cong C_0(G)\text{  and }C_0(X\times G)/I_s \cong C_0(G),
%$$
%induced by the restrictions to  $G\simeq X\rtimes_r G=\{(x, g)\in X\times G, r(g)=x\}$ and $G\simeq X\rtimes_s G=\{(x, g)\in X\times G, s(g)=x\}$ respectively. In the same way, given a morphism $\phi: (A, \theta_A)\rightarrow (B, \theta_B)$ between two $C_0(X)$-algebras, a similar construction yields morphisms $s^*\phi$ and $r^*\phi$ between the lifted $C_0(G)$-algebras.
%
%

\begin{definition}\label{Galgebra}
A $G$-algebra is a $C_0(X)$-algebra $(A, \theta)$ together with an isomorphism $\alpha:s^*A\rightarrow r^*A$ of $C_0(G)$-algebras such that
$$
\pi_{02}^*\alpha=\pi_{01} ^*\alpha \circ \pi_{12}^*\alpha,
$$
where these maps are seen as $C_0(G^{(2)})$-algebra isomorphisms from $\pi_{02}^*s^*A$ to $\pi_{02}^*r^*A$.
\end{definition}

So, we have the following commutative diagram
\[
\begin{CD}
\pi_{12}^*s^*A @> \pi_{12}^*\alpha    >> \pi_{12}^*r^*A \\
@V\pi_{02}^*\alpha VV     @VV \id V\\
\pi_{01}^*r^*A @< \pi_{01}^*\alpha <<\pi_{01}^*s^*A\\
\end{CD}
\]

\begin{remark}
In term of the corresponding u.s.c. field $(A_x)_{x\in X}$, the above isomorphism $\alpha: r^*A\rightarrow s^*A$ of $C_0(G)$-algebras satisfies  the expected relation
$\alpha_{gg'}=\alpha_g\circ \alpha_{g'}$,   for $(g,g')\in G^{(2)}.$
\end{remark}

\begin{definition}
Let $(A, \theta_A, \alpha_A)$ and $(B, \theta_B, \alpha_B)$ be $G$-algebras.
 A morphism  between these two $G$-algebras is a morphism $\phi$ between the $C_0(X)$-algebras  $(A, \theta_A)$ to $(B, \theta_B)$ such that
$$
r^*\phi\circ \alpha^A=\alpha^B\circ s^*\phi.
$$
\end{definition}

So the following diagram is assumed to commute:
\[
\begin{CD}
s^*A @> \alpha^A    >> r^*A \\
@Vs^*\phi VV     @VV r^*\phi V\\
s^*B @> \alpha^B >>r^*B
\end{CD}
\]

%\subsection{Hilbert $G$-modules}
We now review the notion of a Hilbert $G$-module. For the classical material about Hilbert modules over $C^*$-algebras, we refer the reader for instance to \cite{Kasparov} or to the more recent monograph \cite{Lance}.
Let $E$ be a Hilbert $A$-module where $A$ is assumed to be a $G$-algebra with the isomorphism $\alpha:s^*A\rightarrow r^*A$ satisfying the conditions of Definition (\ref{Galgebra}). Then we  may define the fibre of $E$ at $x\in X=G^{(0)}$ as being  $E_x:= E\otimes_A A_x.$ Then $E_x$ is inherits the structure of a Hilbert $A_x$-module.
We define the Hilbert $r^*A$-module, denoted $r^*E$ so that its fibre at $g\in G$ is given by $E_{r(g)}$.
More precisely,
$$
r^*E:= E\otimes_{A}r^*A,
$$
carries a left module action of $A$ through multiplication on the first factor.
Similarly we can define the Hilbert $s^*A$-module $s^*E$. In this situation, the notion of adjointable operator between Hilbert modules over isomorphic $C^*$-algebras is well defined \cite{Lance}. In particular, we shall denote by $\maL_\alpha (s^*E,r^*E)$ the space of adjointable $\alpha$-linear operators. More precisely, using the isomorphism $\alpha$, we endow $r^*E$ with the structure of a Hilbert module over the $C^*$-algebra $s^*A$ and $\maL_\alpha (s^*E, r^*E)$ is the space of adjointable operators between the Hilbert $s^*A$-modules thus obtained. This is the space of linear maps $T:s^*E\rightarrow r^*E$ such that $T(\xi u) = T(\xi) \alpha (u)$ for any $\xi \in s^*E$ and  $u\in s^*A$
and which are adjointable in the sense that there exists an $\alpha^{-1}$-linear operator $T^\sharp: r^*E\rightarrow s^*E$ such that $\left< T(\xi), \eta \right> = \alpha \left( \left<\xi, T^\sharp (\eta) \right>\right)$ for any $\xi\in s^*E$ and $\eta\in r^*E$.

\begin{definition}\
Let $(A, \theta, \alpha)$ be a $G$-algebra as before. A Hilbert $A$-module $E$ is endowed with the structure of a Hilbert $G$-module if we are given a unitary element $V\in \maL_\alpha (s^*E,r^*E)$ such that (setting $V_{ij}=\pi_{ij}^*V$)
$$
V_{01} \circ V_{12} = V_{02} \text{ as elements of } \maL_{C_0(G^{(2)})} (\pi_{02}^*s^*E, \pi_{02}^*r^*E).
$$
\end{definition}

%
%\begin{remark}
%
%Given a Hilbert $C_0(X)$-module $E$, there is a $*$-homomorphism $\theta: C_0(X) \rightarrow Z (\maL_{C_0(X)}(E))$ in the center of the unital $C^*$-algebra of adjointable operators on $E$ given by right multiplication $R_f$, i.e.
%$$
%\theta (f)(\xi) = R_f(\xi) := \xi f, \quad \xi \in E, f\in C_0(X).
%$$
%Notice that $\theta$ is degenerate in general. We define the (unital) $C^*$-algebra $s^*\maL_{C_0(X)}(E)$ in the same way we did for $C_0(X)$-algebras. We get in this way a $*$-homomorphism $s^*\theta : C_0(G) \rightarrow Z  (s^*\maL_{C_0(X)}(E))$. We leave it as an exercise to check that we then have the isomorphisms (DOUBLE-CHECK HERE)
%{I2M: Now I'm not so sure about this statement.}
%$$
%s^*\maL_{C_0(X)}(E)\cong \maL_{C_0(G)}(s^*E),\quad r^*\maL_{C_0(X)}(E)\cong \maL_{C_0(G)}(r^*E)
%$$
%\end{remark}

If we fix a Hilbert $G$-module $(E, V)$ over the $G$-algebra $(A, \theta, \alpha)$ and let $\what{V}: \maL_{C_0(G)}(s^*E)\rightarrow \maL_{C_0(G)}(r^*E)$ be  conjugation by $V$, i.e.
$$
\widehat{V}(T):= V\circ T\circ V^*, \quad\text{ for }T\in \maL_{C_0(G)}(s^*E).
$$
Clearly the operator   $\what{V} (T)$ is then adjointable on the Hilbert $r^*A$-module $r^*E$.

\begin{definition}
An element $T\in \maL_{C_0(X)}(E)$ is called $G$-equivariant if the following equality holds:
$$
\what{V}(s^*T)=r^*T
$$
where $s^*T=T\otimes_{C_0(X),s} \id\in \maL_{C_0(G)}(s^*E)$ and $r^*T=T\otimes_{C_0(X),r} \id\in \maL_{C_0(G)}(r^*E)$.
\end{definition}

The space of $G$-equivariant elements in $\maL_{C_0(X)}(E)$ forms a $C^*$-subalgebra that we denote by $\maL_{C_0(X)}(E)^G$.

We can as well use the maps $\pi_{01}, \pi_{12}$ and $\pi_{02}$ defined above to pull-back one step further the Hilbert $C_0(X)$-module $E$ and get Hilbert modules
$$
\pi_{01}^*(s^*E) = \pi_{12}^*(r^*E), \; \pi_{01}^* r^*E = \pi_{02}^* r^*E \;\text{ and } \; \pi_{12}^* s^*E = \pi_{02}^* s^*E.
$$
Then the pull-back of the transformation $\what{V}$ gives
$$
\maL_{C_0(G^{(2)})}(\pi_{12}^*s^*E) \stackrel{\what{V}_{12}}{\longrightarrow}\maL_{C_0(G^{(2)})}(\pi_{12}^*r^*E) = \maL_{C_0(G^{(2)})}(\pi_{01}^*s^*E) \stackrel{\what{V}_{01}}{\longrightarrow} \maL_{C_0(G^{(2)})}(\pi_{01}^*r^*E)
$$
and also
$$
\maL_{C_0(G^{(2)})}(\pi_{12}^*s^*E) = \maL_{C_0(G^{(2)})}(\pi_{02}^*s^*E) \stackrel{\what{V}_{02}}{\longrightarrow}  \maL_{C_0(G^{(2)})}(\pi_{02}^*r^*E) = \maL_{C_0(G^{(2)})}(\pi_{01}^*r^*E).
$$
From the properties of the unitary $V$, we deduce that the map $\what{V}$ satisfies the cocycle condition
$$
\what{V}_{01}\circ \what{V}_{12}= \what{V}_{02}.
$$

%\subsection{$G$-equivariant representations}

To end this appendix, we say some words about $G$-equivariant representations.
\begin{definition}
Let $(A, \theta)$ be a $C_0(X)$-algebra. A $C_0(X)$-representation of $A$ is given by a Hilbert $C_0(X)$-module $E$ together with a $*$-homomorphism $\pi: A \rightarrow \maL_{C_0(X)}(E)$ such that:
$$
\pi(\theta(f)(a))(e) = \pi(a)(ef) \quad\text{ for }f\in C_0(X), a \in A\text{ and } e\in E.
$$
\end{definition}

So,  equivalently the $*$-representation $\pi$ must satisfy the condition
$$
\pi\circ \theta(f)=R_f \circ \pi: A \rightarrow \maL_{C_0(X)}(E),\quad \text{ for any }f\in C_0(X).
$$
Here $R_f$ is the (adjointable) operator implementing the right module action of $C_0(X)$ on $E$. It is clear by definition that any $C_0(X)$-representation $\pi$ gives rise to a field of {{representations}} in Hilbert spaces $(\pi_x: A_x \to \maL (E_x))_{x\in X}$, where $(E_x)_{x\in X}$ is the field of Hilbert spaces over $X$ which is associated with $E$.

\begin{definition}[$G$-equivariant representation]\
Let $(A, \theta, \alpha)$ be a $G$-algebra and let $(E, V)$ be a Hilbert $G$-module. A $C_0(X)$-representation $\pi: A \rightarrow \maL_{C_0(X)}(E)$ is  $G$-equivariant  if the following diagram commutes:

\[
\begin{CD}
s^*A @> \alpha    >> r^*A \\
@Vs^*\pi VV     @VV r^*\pi V\\
\maL_{C_0(G)}(s^*E) @> \hat{V} >> \maL_{C_0(G)}(r^*E)
\end{CD}
\]
\end{definition}

In terms of the u.s.c. field of $C^*$-algebras associated with $A$ and the u.s.c. field of Hilbert modules associated with $E$, the  equivariance property can be written  as expected
$$
\pi_{r(g)} [\alpha_g(x)] = V_g \circ \pi_{s(g)} (x) \circ V_g^{*}, \text{ for  } g\in G \text{ and } x\in A_{s(g)},
$$

\section{Continuous fields of operators}\label{AppendixB}

%\subsection{Continuous fields of bounded operators}

Let $X$ be a paracompact locally compact Hausdorff space. For each $x \in X$, let $H_x$ be a complex separable Hilbert space.
\begin{definition}[Continuous field of Hilbert spaces, compare cf. \cite{DixmierC*}, Definition 10.1.2]
A continuous field of Hilbert spaces is a (complex) linear subspace $F {{\subset}} \prod_{x\in X} H_x$ such that:
\begin{enumerate}
\item(totality) the collection $\{u(x): u \in F\} \subset H_x\}$ is dense in $H_x$,
\item(norm continuity) for every $u \in F$, the map $x\mapsto ||u(x)||$ is a continuous function vanishing at infinity,
\item(closure under uniform local approximability) given $v\in \prod_{x\in X} H_x$, if for each $\epsilon>0$ and each $x_0\in X$ there exists an element $u \in F$ and a neighbourhood $U_{x_0}$ of $x_0$ such that $||v(x) - u(x)||<\epsilon, \forall x \in U_{x_0}$, then $v \in F$.
%\item(uniform local approximability) given $v\in \prod_{x\in X} H_x$, if for each $\epsilon>0$ and $x_0 \in X$, there exists a neighbourhood $U_{x_0}$ of $x$ and an element $u \in F$ such that $||v(x) - u(x)||<\epsilon, \forall x \in U_{x_0}$, then $t \in F$.
\end{enumerate}
\end{definition}

%\begin{remark}
% As a consequence of Michael's selection theorem, such a continuous field induces a submodule of $C_0(X, H)$ for some infinite dimensional separable Hilbert space $H$, and for each $x\in X$, the Hilbert space $H_x$ is a closed subspace of $H$, cf. \cite{DixmierC*}, Chapter 10.
%\end{remark}

%Let $u,v \in F$. By the Parallelogram law, the map $x \mapsto <u(x), v(x)>$ is in $C_0(X)$, so we get an element $<u,v>$ of $C_0(X)$ by the formula
%$$
%<u,v>(x) = <u(x), v(x)>
%$$
%
%The $C^*$ and scalar norms on $\maH$ are given in a standard way \cite{Takemoto}.
%
%\subsubsection{Continuous field of bounded operators}

\begin{definition}[Continuous field of bounded operators, cf. \cite{Phillips}, Chapter 1]
A continuous field of bounded operators on $F$ is a family $t= \{t_x\}_{x\in X} \in \prod_{x\in X} B(H_x)$ such that:
\begin{enumerate}
\item (uniform bound in norm) $||t||:=\sup_{x\in X} ||t_x|| <\infty$
\item(strong-* continuity) for $u \in F$, the elements $tu= \{t_xu_x\}_{x\in X}$  and $t^*u = \{t^*_xu_x\}_{x\in X}$ belong to $F$.
\end{enumerate}
\end{definition}

A continuous field of Hilbert spaces $F$ over $X$ gives rise to a full Hilbert $C_0(X)$-module $\maH$.
We shall denote the collection of such continuous fields of bounded operators by $\maB$.

\begin{remark} (cf. \cite{Phillips}, Chapter 1)\label{Adjointability_Continuousfields}
Let $t\in \maB$. Then $t$ induces an adjointable operator on $T$ on the Hilbert $C_0(X)$-module $\maH$. In turn, any element $T\in \maL_{C_0(X)}(\maH)$ defines a continuous field of bounded operators $t \in \maB$.
\end{remark}

The following result is standard and will be used later.

\begin{lemma}[Dense range]\label{denserange}
Let $\S=(S_x)_{x\in X}$ be a continuous family of bounded operators on $F$, such that $\Range(S_x)$ is dense in $H_x$ for each $x\in X$. Then the adjointable operator $S$ induced by $\S$ has dense range on $\maH$.
\end{lemma}

For a Banach algebra $B$ we denote the spectrum of an element $a\in B$ by $Sp(a)$.
Notice that if $t=\{t_x\}_{x\in X}\in \maB$  induces the operator $T \in \maL_{C_0(X)}(\maH)$, then, we have for each $x\in X$, $Sp(t_x)\subseteq Sp(T)$.

%\begin{proof}
%Let $\lambda \in  \rho(T)= \C \setminus Sp(T)$, so that $T-\lambda I $ is adjointably invertible. We will show that $\lambda \in \rho(t_x)$ for each $x \in X$, i.e. that $t_x- \lambda I$ is invertible. By the open mapping theorem, it suffices to check that bijectivity of $(T-\lambda I)$ implies bijectivity of $(t_x-\lambda I)$.
%
%Injectivity: Let $(t_x-\lambda I)\xi =0$ for $\xi \in H_x$. Appealing to \cite{DixmierC*}, Proposition 10.1.10, there exists an element $\eta \in F$ such that $\eta(x)=\xi$, so we get $(t_x-\lambda I)\eta(x)=0$. The element $\eta \in F$ can be viewed as an element of $\maH$, which we continue to denote by $\eta$. Therefore we have,
%$$
%[(T-\lambda)\eta](x)= (t_x- \lambda I)(\eta(x))= (t_x-\lambda I)\xi =0
%$$
%
%Let $\epsilon >0$ be arbitrary. Continuity of $(T-\lambda I)\eta$ implies that there exists an open neighbourhood $V_x$ of $x$ such that
%$ |[(T-\lambda I)\eta](x)| <\epsilon$. Let $g \in C_c(X, [0,1])$ be a function which is equal to 1 on an open neighbourhood $U_x \subset V_x$ and which vanishes outside $V_x$. Then $\eta' = g \eta$ agrees with $\eta$ on $U_x$ and belongs to $F$ by \cite{DixmierC*}, Proposition 10.1.9. So we get $||(T-\lambda I)\eta'||<\epsilon$. Since $\epsilon$ was arbitrary this implies $(T-\lambda I)\eta'=0$, and since $(T-\lambda I)$ is injective, $\eta'=0$. In particular $\eta'(x)=\eta(x)=\xi=0$. So $(t_x-\lambda I)$ is injective.
%
%Surjectivity is proved in a similar fashion as above and is left to the reader as an exercise.
%\end{proof}

\begin{lemma}\label{compatibilitybounded}
Let $T\in \maL_{C_0(X)}(\maH)$ be self-adjoint, inducing an element $t\in \maB$ such that each $t_x$ is self-adjoint. Let $f \in C_b(\R)$, then the operator $f(T) \in \maL_{C_0(X)}(\maH)$ given by the continuous functional calculus induces a continuous field of bounded operators $s \in \maB$ such that $s_x =f(t_x)$.
\end{lemma}
\begin{proof}
The statement is obvious for polynomials, and for $f\in C_b(Sp(T))$ one uses a standard approximation argument.
\end{proof}

%\subsection{Continuous fields of closed operators}
We now consider fields of closed (unbounded) operators. start with the following proposition/definition.

\begin{proposition}\cite{Skandaliscourse, Lance}\label{regularop}
Let $T$ be a densely defined, closed, unbounded operator on $\maH$, and suppose that $T^*$ is also densely-defined. Then the following conditions are equivalent:
\begin{itemize}
\item $I+T^*T$ has dense image.
\item $I+T^*T$ is surjective.
\item The graph of $T$, denoted $G(T),$ is an orthocomplemented submodule of $\maH\oplus \maH$, with $G(T)^\perp = \sigma G(T^*)$, where $\sigma: \maH\oplus \maH\rightarrow \maH\oplus \maH$ is the map $(x,y)\mapsto (y,-x)$.
\item The projection $p$ onto $G(T)$ is a self-adjoint idempotent in $\maL(\maH\oplus \maH)$.
\end{itemize}
\end{proposition}

%\begin{proof}
%This is a standard result, see for instance  \cite{Skandaliscourse}[Lemma 15.2] and \cite{Lance}[Theorem 9.3, Proposition 9.5].
%\end{proof}
An operator satisfying one of the equivalent properties in Proposition \ref{regularop} is called a regular operator. To sum up an operator $T$ is regular if it is densely defined as well as its adjoint and if its graph ia an orthocomplemented submodule. If $T$ is regular then so is $T^*$ and one has $(T^*)^*=T$. It is also {{worth pointing}} out that for such regular operator $T$, the operator $T^*T$ is a self-adjoint regular operator whose spectrum is contained in $\R_+$. Therefore for any continuous bounded function $f:\R_+\to \C$, there is a well defined adjointable operator $f(T^*T)$ given by the spectral theorem, see again \cite{Skandaliscourse}. We then set:
$$
Q(T) := T W(T) \text{ with }  W(T) := (I+T^*T)^{-1/2} \text{ and we define similarly  }Q(T^*) \text{ and } W(T^*).
$$
Notice for instance  that $Q(T^*)=Q(T)^*$, $W(T^*)= W(T)^*$ and $(I-Q(T^*)Q(T))^{1/2}=W(T)$. See again \cite{Skandaliscourse}.
%\begin{remark}\label{graphprojection} The projection $p$ in the proposition above can be given explicitly as:
%$$
%p=  \begin{bmatrix}   W(T)^2 & W(T)Q(T)^*\\ Q(T)W(T)& Q(T)Q(T)^* \end{bmatrix}
%$$
%Thus it can be described solely through the adjointable operators $W(T)$ and $Q(T)$.
%\end{remark}
Recall that we have fixed a continuous field of Hilbert spaces $F\subset  \prod_{x\in X} H_x$. Let $(D_x)_{x\in X}$ be a family of closed unbounded operators on the family of Hilbert spaces $(H_x)_{x\in X}$. We denote the dense domains of $D_x$ by $\Dom(D_x)$.

\begin{definition} \cite{Takemoto}\label{Continuousclosed}
The field $\D: =(D_x)_{x\in X}$ of closed operators is called a continuous field  if the field $(Q(D_x))_{x\in X}$  determines a continuous
field of bounded operators on the continuous field of Hilbert spaces $F$.
\end{definition}

\begin{remark}\
\begin{enumerate}
\item One could also ask for the equivalent condition that the field of projections $(p_x)_{x\in X}$ onto the graphs of $(D_x)_{x\in X}$ be a continuous field of bounded operators.
\item When $\D$ is a self-adjoint family, then it is continuous if and only if the two fields $((D_x\pm i)^{-1})_{x\in X}$ are strongly continuous.
\end{enumerate}
\end{remark}

%Let $(D_x)_{x\in X}$ be a family of closed, self-adjoint unbounded operators on  $(H_x)_{x\in X}$ as before. We denote the dense domains of $D_x$ by $\Dom(D_x)$.

%\begin{definition} \cite{Takemoto}\label{Continuousclosed}
%The family $\D: =(D_x)_{x\in X}$ is called a continuous family (of closed, self-adjoint unbounded) operators if the families $\{Q(D_x)\}_{x\in X}$ and $\{W(D_x)\}_{x\in X}$ determine continuous families of bounded operators on the continuous field of Hilbert spaces $F$.
%\end{definition}
%
\begin{lemma}
A continuous field $\D=(D_x)_{x\in X}$ of closed self-adjoint  operators is well-defined on $F$ and induces a regular self-adjoint operator on the associated Hilbert module $\maH$.
\end{lemma}

\begin{proof}\
The field $(W(D_x):=(I+ D_x^2)^{-1/2})_{x\in X}$ is then also continuous and by Lemma \ref{denserange}, the adjointable operator $\maR$ induced by this continuous field $(W(D_x))_{x\in X}$ has dense range in $\maH$, which is thus a dense submodule of $\maH$. Let $u \in \Im (\maR)$, then each $u_x \in \Im(W(D_x))=\Dom(D_x)$, and there exists an element $v\in F$ such that $u_x = W(D_x)v_x$. Therefore, $D_xu_x = Q(D_x)v_x$, which is continuous by hypothesis. Since $(1-Q(D_x)^2)^{1/2} = W(D_x)$, the regular operator $\maS$ on the Hilbert module $\maH$ associated with the continuous field $(Q(D_x))_{x\in X}$ satisfies that $(1-\maS^2)^{1/2}$ has dense range in $\maH$ with the obvious relation $||S||\leq 1$. Therefore, applying \cite{Skandaliscourse}[Th\'{e}or\`{e}me 15.10], we deduce that $\maS$ corresponds uniquely to a regular operator $\maD$ on $\maH$ such that $\maS= Q(\maD)$.
% First let us show that for $u$ in a dense submodule of $\maH$ (viewed in $F$), the family $\D u \in F$. By the functional calculus of closed self-adjoint  operators, for each $x\in X$, we know that the operators $Q(D_x)$ and $W(D_x)$ are  bounded operators, and that $\Dom(D_x)= \Im(W(D_x))$ and $\Im(D_x)= \Im(Q(D_x))$.
%In particular, $\Im(W(D_x))$ is dense in $H_x$ for each $x\in X$. Thus
%Now let us consider the self-adjoint operator $\maS$ induced by the continuous field of bounded operators $Q(D_x)$. By the general theory of unbounded operators, each operator $(1-Q(D_x)^2)^{1/2}$ has dense range in $H_x$ M2I: I DON'T UNDERSTAND THE PREVIOUS SENTENCE, ISN'T THIS SIMPLY $W(D_x)$?. We apply Lemma \ref{denserange} again, this time to the field $(1-Q(D_x)^2)^{1/2}$, which is continuous by the functional calculus of continuous families of bounded operators. By the compatibility of functional calculi (cf.  Lemma (\ref{compatibilitybounded})), we get that the operator $(1-\maS^2)^{1/2}$ has dense range in $\maH$. It is also clear that $||S||\leq 1$. Therefore, $\maS$ corresponds uniquely to a regular operator $\maD$ on $\maH$ such that $\maS= Q(\maD)$, cf. \cite{Skandaliscourse}, Th\'{e}or\`{e}me 15.10.
\end{proof}

\begin{lemma}[Compatibility of functional calculi]\label{compatibility}
Let $\D=(D_x)_{x\in X}$ be a continuous family of closed self-adjoint operators. Let $\maD$ be the regular operator on $\maH$ induced by $\D$. Then for $f\in C_b(\R)$ the adjointable operator $f(\maD) \in \maL_{C_0(X)}(\maH)$, defined by the functional calculus of regular self-adjoint operators, corresponds to the (continuous) family $\{f(D_x)\}_{x\in X}$ of bounded operators.
\end{lemma}

\begin{proof}
This follows from the compatibility of functional calculi of continuous families of bounded operators and adjointable operators (cf. Lemma \ref{compatibilitybounded}), since the functional calculus of regular operators is defined via the $Q$-transform.
\end{proof}

\end{document}